\documentclass[reqno]{amsart}
\usepackage{amsmath,amssymb}
\usepackage[mathscr]{euscript}
\usepackage{graphicx}
\usepackage{enumitem}
\usepackage[english]{babel}

\usepackage[small]{caption}

\usepackage{tikz}

\makeatletter
\@addtoreset{equation}{section}
\makeatother

\newtheorem{theorem}{Theorem}[section]
\newtheorem{lemma}[theorem]{Lemma}
\newtheorem{proposition}[theorem]{Proposition}
\newtheorem{corollary}[theorem]{Corollary}

\newtheorem{remark}[theorem]{Remark}
\newtheorem{assumption}[theorem]{Assumption}

\numberwithin{equation}{section}

\newcommand{\mc}[1]{{\mathcal #1}}
\newcommand{\ms}[1]{{\mathscr #1}}
\newcommand{\mf}[1]{{\mathfrak #1}}
\newcommand{\mb}[1]{{\mathbf #1}}
\newcommand{\bb}[1]{{\mathbb #1}}
\newcommand{\bs}[1]{{\boldsymbol #1}}

\newcounter{as}[section]

\newtheorem{asser}[as]{Assertion}

\newcommand{\<}{\langle}
\renewcommand{\>}{\rangle}
\renewcommand{\Cap}{{\rm cap}}

\begin{document}

\title[Metastability of finite state Markov chains] {Metastability of
  finite state Markov chains: a recursive procedure to identify slow
  variables for model reduction}

\author{C. Landim, T. Xu}

\address{\noindent IMPA, Estrada Dona Castorina 110, CEP 22460-320 Rio
  de Janeiro, Brasil. \newline
\noindent CNRS UMR 6085, Universit\'e de Rouen, Avenue de
  l'Universit\'e, BP.12, Technop\^ole du Madril\-let, F76801
  Saint-\'Etienne-du-Rouvray, France.  \newline e-mail: \rm
  \texttt{landim@impa.br} }

\address{\noindent IMPA, Estrada Dona Castorina 110, CEP 22460 Rio de
 Janeiro, Brasil.  \newline e-mail: \rm
 \texttt{xu@impa.br} }

\begin{abstract}
  Consider a sequence $(\eta^N(t) :t\ge 0)$ of continuous-time,
  irreducible Markov chains evolving on a fixed finite set $E$,
  indexed by a parameter $N$. Denote by $R_N(\eta,\xi)$ the jump rates
  of the Markov chain $\eta^N_t$, and assume that for any pair of
  bonds $(\eta,\xi)$, $(\eta',\xi')$ $\arctan
  \{R_N(\eta,\xi)/R_N(\eta',\xi')\}$ converges as
  $N\uparrow\infty$. Under a hypothesis slightly more restrictive
  (cf. \eqref{mhyp} below), we present a recursive procedure which
  provides a sequence of increasing time-scales $\theta^1_N, \dots,
  \theta^{\mf p}_N$, $\theta^j_N \ll \theta^{j+1}_N$, and of
  coarsening partitions $\{\ms E^j_1, \dots, \ms E^j_{\mf n_j},
  \Delta^j\}$, $1\le j\le \mf p$, of the set $E$.  Let $\phi_j: E \to
  \{0,1, \dots, \mf n_j\}$ be the projection defined by $ \phi_j(\eta)
  = \sum_{x=1}^{\mf n_j} x \, \mb 1\{\eta \in \ms E^j_x\}$.  For each
  $1\le j\le \mf p$, we prove that the hidden Markov chain $X^j_N(t) =
  \phi_j(\eta^N(t\theta^j_N))$ converges to a Markov chain on $\{1,
  \dots, \mf n_j\}$.
\end{abstract}

\keywords{Metastability, Markov chains, slow variables, model reduction}

\maketitle

\section{Introduction}

This article has two motivations.  On the one hand, the metastable
behavior of non-reversible Markovian dynamics has attracted much
attention recently \cite{msv09, mo13, blm13, l14, m14, cns14, bg15,
  cg15, fmns15, fmnss15}. On the other hand, the emergence of large
complex networks gives a particular importance to the problem of data
and model reduction \cite{lv14, ce14, ag14}. This issue arises in as
diverse contexts as meteorology, genetic networks or protein folding,
and is very closely related to the identification of slow variables, a
fundamental tool in decreasing the degrees of freedom \cite{sekc15}.

Not long ago, Beltr\'an and one of the authors introduced a general
approach to derive the metastable behavior of continuous-time Markov
chains, particularly convenient in the presence of several valleys
with the same depth \cite{bl2, bl7, bl9}. In the context of finite
state Markov chains \cite{bl4}, it permits to identify the slow
variables and to reduce the model and the state space.

More precisely, denote by $E$ a finite set, by $\eta^N_t$ a sequence
of $E$-valued continuous-time, irreducible Markov chains, and by $\ms
E_1, \dots, \ms E_{\mf n}, \Delta$ a partition of the set $E$. Let
$\ms E = \cup_{1\le x \le \mf n} \ms E_x$ and let $\phi_{\ms
  E}: E \to \{0,1, \dots, \mf n\}$ be the projection defined by
\begin{equation*}
\phi_{\ms E}(\eta) \;=\; \sum_{x=1}^{\mf n} x \, \mb 1\{\eta
\in \ms E_x\} \;.
\end{equation*}
In general, $X_N(t) = \phi_{\ms E}(\eta^N_t)$ is not a Markov chain,
but only a hidden Markov chain.  We say that $\phi_{\ms E}$ is a slow
variable if there exists a time-scale $\theta_N$ for which the
dynamics of $X_N(t\theta_N)$ is asymptotically Markovian.

The set $\Delta$ plays a special role in the partition, separating the
sets $\ms E_1, \dots, \ms E_{\mf n}$, called here valleys. The chain
remains a negligible amount of time in the set $\Delta$ in the
time-scale $\theta_N$ at which the slow variable evolves.

Slow variables provide an efficient mechanism to contract the state
space and to reduce the model in complex networks, as it allows to
represent the original evolution through a simple Markovian chain
$X_N(t)$ which takes value in a much smaller set, without losing the
essential features of the dynamics.  It may also reveal aspects of the
dynamics which may not be apparent at first sight.

When the number of sets in the partition is reduced to $2$, $\mf n=2$,
and the Markov chain which describes the asymptotic behavior of the
slow variable has one absorbing point and one transient point, the
chain presents a metastable behavior. In a certain time-scale, it
remains for an exponential time on a subset of the state space after
which it jumps to another set where it remains for ever. By extension,
and may be inapropriately, we say that the chain $\eta^N_t$ exhibits a
metastable behavior among the valleys $\ms E_1, \dots, \ms E_{\mf n}$
in the time-scale $\theta_N$ whenever we prove the existence of a slow
variable.

We present in this article a recursive procedure which permits to
determine all slow variables of the chain. It provides a sequence of
time-scales $\theta^1_N, \dots, \theta^{\mf p}_N$ and of partitions
$\{\ms E^j_1, \dots, \ms E^j_{\mf n_j}, \Delta_j\}$, $1\le j\le \mf
p$, of the set $E$ with the following properties.

\begin{itemize}
\item The time-scales are increasing: $\lim_{N\to\infty} \theta^j_N
  /\theta^{j+1}_N =0$ for $1\le j<\mf p$. This relation is represented
  as $\theta^j_N \ll \theta^{j+1}_N$.

\item The partitions are coarser. Each set of the $(j+1)$-th partition
  is obtained as a union of sets in the $j$-th partition. Thus $\mf
  n_{j+1} < \mf n_j$ and for each $a$ in $\{1, \dots, \mf n_{j+1}\}$,
  $\ms E^{j+1}_a = \cup_{x\in A} \ms E^{j}_x$ for some subset $A$ of
  $\{1, \dots, \mf n_j\}$.

\item The sets $\Delta_j$, which separates the valleys, increase:
  $\Delta_j \subset \Delta_{j+1}$. Actually, $\Delta_{j+1}=\Delta_j
  \cup _{x\in B} \ms E^{j}_x$ for some subset $B$ of $\{1, \dots, \mf
  n_j\}$.

\item The projection $\Psi^j_N(\eta) \;=\; \sum_{1\le x\le \mf n_j} x \, \mb
  1\{\eta \in \ms E^j_x\} \;+\; N \, \mb 1\{\eta \in \Delta_j\}$ is a
  slow variable which evolves in the time-scale $\theta^j_N$.
\end{itemize}

We prove three further properties of the partitions $\{\ms E^j_1,
\dots, \ms E^j_{\mf n_j}, \Delta_j\}$.

\begin{itemize}
\item As mentioned above, the time the chain remains in the set
  $\Delta_j$ in the time-scale $\theta^j_N$ is negligible. We refer to
  condition (H3) below for a mathematical formulation of this
  assertion.

\item Starting from any configuration in $\ms E^j_x$, the chain
  $\eta^N_t$ attains the set $\cup_{y\not = x} \ms E^j_y$ at a time
  which is asymptotically exponential in the time-scale
  $\theta^j_N$ (cf. Remark \ref{rm1b}).

\item With a probability asymptotically equal to $1$, the chain
  $\eta^N_t$ visits all points of the set $\ms E^j_x$ before hitting 
another set $\ms E^j_y$ of the partition. In the terminology of Freidlin
  and Wentzell \cite{fw}, the sets of the first partition, denoted by
  $\ms E^1_x$,  are cycles while the set of the following partitions
  are cycles of cycles. 
\end{itemize}

These results have been proved in \cite{bl4} for finite state
reversible Markovian dynamics. We remove in this article the
assumption of reversibility and we simplify some proofs.

In contrast with other approaches \cite{mnos04, ov1, ev06, msv09,
  cns14, fmns15, fmnss15}, we do not describe the tube of typical
trajectories in a transition between two valleys, nor do we identify
the critical configurations which are visited with high probability in
such transitions. 

The arguments presented here have been designed for sequences of
Markov chains. The examples we have in mind are zero-temperature
limits of non-reversible dynamics in a finite state space. It is not
clear whether the analysis can be adapted to handle the case of a
single fixed dynamics as in \cite{ce14, lv14, ag14}.

The approach presented in this article is based on a multiscale
analysis. The sequence of increasing time-scales is defined in terms
of the depth of the different valleys. In this sense, the method is
similar to the one proposed by Scoppola in \cite{s93}, and developed
by Olivieri and Scoppola \cite{os95, os96}, but it does not require
the valleys to have exponential depth, nor the jump rates to be
expressed in terms of exponentials. Actually, one of its main merit is
that it relies on a minimal hypothesis, presented in \eqref{mhyp}
below, which is very easy to check since it is formulated only in
terms of the jump rates. 

The article is organized as follows. In Section \ref{not} we state the
main results. In the following three sections we introduce the tools
needed to prove these results, which is carried out in the last three
sections.

\section{Notation and main results}
\label{not}

Consider a finite set $E$. The elements of $E$ are called
configurations and are denoted by the Greek letters $\eta$, $\xi$,
$\zeta$. Consider a sequence of continuous-time, $E$-valued,
irreducible Markov chains $\{\eta^N_t : t\ge 0\}$.  Denote the jump
rates of $\eta^N_t$ by $R_N(\eta,\xi)$, and by $\mu_N$ the unique
invariant probability measure.

Denote by $D(\bb R_+, E)$ the space of right-continuous functions $x:
\bb R_+ \to E$ with left-limits endowed with the Skorohod topology,
and by $\bb P_\eta=\bb P^N_\eta$, $\eta\in E$, the probability measure
on the path space $D(\bb R_+, E)$ induced by the Markov chain
$\eta^N_t$ starting from $\eta$. Expectation with respect to $\bb
P_\eta$ is represented by $\bb E_\eta$.

Denote by $H_{\ms A}$, $H^+_{\ms A}$, ${\ms A}\subset E$, the
hitting time and the time of the first return to ${\ms A}$:
\begin{equation}
\label{201}
H_{\ms A} \;=\; \inf \big \{t>0 : \eta^N_t \in {\ms A} \big\}\;,
\quad
H^+_{\ms A} \;=\; \inf \big \{t>\tau_1 : \eta^N_t \in {\ms A} \big\}\; ,  
\end{equation}
where $\tau_1$ represents the time of the first jump of the chain
$\eta^N_t$: $\tau_1 = \inf\{t>0 : \eta^N_t \not = \eta^N_0\}$.  

Denote by $\lambda_N(\eta)$, $\eta\in E$, the holding rates of the
Markov chain $\eta^N_t$ and by $p_N(\eta,\xi)$, $\eta$, $\xi\in E$,
the jump probabilities, so that $R_N(\eta,\xi) = \lambda_N(\eta)
p_N(\eta,\xi)$.  For two disjoint subsets $\ms A$, $\ms B$ of $E$, denote by
$\Cap_N(\ms A, \ms B)$ the capacity between $\ms A$ and $\ms B$:
\begin{equation}
\label{28}
\Cap_N(\ms A, \ms B) \;=\; \sum_{\eta\in \ms A} \mu_N(\eta)\, \lambda_N(\eta) 
\, \bb P_{\eta} [H_{\ms B} < H^+_{\ms A}]\;.
\end{equation}

Consider a partition $\ms E_1, \dots, \ms E_{\mf n}$, $\Delta$ of the
set $E$, which does not depend on the parameter $N$ and such that $\mf
n\ge 2$. Fix two sequences of positive real numbers $\alpha_N$,
$\theta_N$ such that $\alpha_N\ll \theta_N$, where this notation
stands for $\lim_{N\to\infty} \alpha_N/\theta_N = 0$.

Let $\ms E = \cup_{x\in S} \ms E_x$, where $S=\{1, \dots, \mf n\}$.
Denote by $\{\eta^{\ms E}_t: t\ge 0\}$ the trace of $\{\eta^N_t: t\ge
0\} $ on $\ms E$, and by $R^{\ms E}_N : \ms E \times \ms E\to \bb R_+$
the jump rates of the trace process $\eta^{\ms E}_t$. We refer
to Section 6 of \cite{bl2} for a definition of the trace
process. Denote by $r^{\ms E}_N(\ms E_x,\ms E_y)$ the mean rate at
which the trace process jumps from $\ms E_x$ to $\ms E_y$:
\begin{equation}
\label{b04}
r^{\ms E}_N(\ms E_x,\ms E_y) \; = \; \frac{1}{\mu_N(\ms E_x)}
\sum_{\eta\in\ms E_x} \mu_N(\eta) 
\sum_{\xi\in\ms E_y} R^{\ms  E}_N(\eta,\xi) \;.
\end{equation}
Assume that for every $x\not =y\in S$,
\begin{equation*}
\tag*{\bf (H1)}
\begin{split}
& r_{\ms E}(x,y) \;:=\;  \lim_{N\to\infty} \theta_N
\, r^{\ms E}_N(\ms E_x,\ms E_y) \;\in\; \bb R_+\;, \\
&\quad \text{and that}\quad \sum_{x\in S} \sum_{y\not = x} r_{\ms
  E}(x,y) >0\;.
\end{split}
\end{equation*}  
The symbol $:=$ in the first line of the previous displayed equation
means that the limit exists, that it is denoted by $r_{\ms E}(x,y)$,
and that it belongs to $\bb R_+$. This convention is used throughout
the article.

Assume that for every $x\in S$ for which $\ms E_x$ is not a
singleton and for all $\eta\not =\xi\in \ms E_x$,
\begin{equation*}
\tag*{\bf (H2)}
\liminf_{N\to \infty} \alpha_N\, \frac{\Cap_N(\eta,\xi)}
{\mu_N(\ms E_x)}\;>\;0\; .
\end{equation*}

Finally, assume that in the time scale $\theta_N$ the chain remains
a negligible amount of time outside the set $\ms E$:
For every $t>0$,
\begin{equation*}
\tag*{\bf (H3)}
\lim_{N\to \infty} \max_{\eta\in E} \, \bb E_\eta \Big[
\int_0^t \mb 1\{ \eta^N_{s\theta_N} \in \Delta\} \, ds  \Big] \;=\; 0\;. 
\end{equation*}

Denote by $\Psi_N : E\to \{1, \dots, \mf n, N \}$ the projection defined
by $\Psi_N(\eta) = x$ if $\eta\in \ms E_x$, $\Psi_N(\eta) = N$,
otherwise:
\begin{equation*}
\Psi_N(\eta) \;=\; \sum_{x\in S} x \, \mb 1\{\eta \in \ms E_x\} 
\;+\; N \, \mb 1 \{\eta \in \Delta \}\;.
\end{equation*}
Recall from \cite{l-soft} the definition of the soft topology.  

\begin{theorem}
\label{s01}
Assume that conditions {\rm (H1)--(H3)} are in force.  Fix $x\in S$
and a configuration $\eta\in\ms E_x$. Starting from $\eta$, the
speeded-up, hidden Markov chain $\bs X_N(t) =
\Psi_N\big(\eta^N(\theta_N t)\big)$ converges in the soft topology to
the continuous-time Markov chain $X_{\ms E}(t)$ on $\{1, \dots, \mf
n\}$ whose jump rates are given by $r_{\ms E}(x,y)$ and which starts
from $x$.
\end{theorem}

This theorem is a straightforward consequence of known results. We
stated it here in sake of completeness and because all the analysis of
the metastable behavior of $\eta^N_t$ relies on it.


\begin{remark}
\label{rm1a}
Theorem \ref{s01} states that in the time scale $\theta_N$, if we just
keep track of the set $\ms E_x$ where $\eta^N_t$ is and not of the
specific location of the chain, we observe an evolution on the set $S$
close to the one of a continuous-time Markov chain which jumps from
$x$ to $y$ at rate $r_{\ms E}(x,y)$.
\end{remark}

\begin{remark}
\label{rm1d}
The function $\Psi_N$ represents a slow variable of the chain. Indeed,
we will see below that the sequence $\alpha^{-1}_N$ stands for the
order of magnitude of the jump rates of the chain. Theorem \ref{s01}
states that on the time scale $\theta_N$, which is much longer than
$\alpha_N$, the variable $\Psi_N(\eta^N_t)$ evolves as a Markov
chain. In other words, under conditions (H1)--(H3), one still observes
a Markovian dynamics after a contraction of the configuration space
through the projection $\Psi_N$. Theorem \ref{s01} provides therefore
a mechanism of reducing the degrees of freedom of the system, keeping
the essential features of the dynamics, as the ergodic properties.
\end{remark}

\begin{remark}
\label{rm1b}
It also follows from assumptions (H1)--(H3) that the exit time from a
set $\ms E_x$ is asymptotically exponential. More precisely, let
$\breve{\ms E}_x$, $x\in S$, be the union of all set $\ms E_y$ except
$\ms E_x$:
\begin{equation}
\label{27}
\breve{\ms E}_x \;=\; \bigcup_{y\not = x} {\ms E}_y\;.
\end{equation} 
For every $x\in S$ and $\eta\in\ms E_x$, under $\bb P_\eta$ the
distribution of $H_{\breve{\ms E}_x}/\theta_N$ converges to an
exponential distribution.
\end{remark}

\begin{remark}
\label{rm1c}
Under the assumptions (H1)--(H3), the sets $\ms E_x$ are cycles in the
sense of \cite{fw}. More precisely, for every $x\in S$ for which
$\ms E_x$ is a not a singleton, and for all $\eta\not = \xi\in\ms
E_x$, 
\begin{equation*}
\lim_{N\to\infty} \bb P_\eta \big[ H_{\xi} < H_{\breve{\ms E}_x}\big]
\;=\; 1\;.  
\end{equation*}
This means that starting from $\eta\in\ms E_x$, the chain visits all
configurations in $\ms E_x$ before hitting the set $\breve{\ms E}_x$.
\end{remark}

\subsection{The main assumption}
\label{ss01}

We present in this subsection the main and unique hypothesis made on
the sequence of Markov chains $\eta^N_t$.  Fix two configurations
$\eta\not =\xi\in E$.  We assume that the jump rate from $\eta$ to
$\xi$ is either constant equal to $0$ or is always strictly positive:
\begin{equation*}
R_N(\eta,\xi) \;=\; 0 \;\;\text{for all $N\ge 1$}\;\;\text{or}\;\;
R_N(\eta,\xi) \;>\; 0 \;\;\text{for all $N\ge 1$}\;.
\end{equation*}
This assumption permits to define the set of ordered bonds of $E$,
denoted by $\bb B$, as the set of ordered pairs $(\eta,\xi)$ such that
$R_N(\eta,\xi)>0$:
\begin{equation*}
\bb B \;=\; \big \{(\eta,\xi) \in E\times E : \eta\not = \xi \,,\,
R_N(\eta,\xi)>0 \big\}\;.
\end{equation*}
Note that the set $\bb B$ does not depend on $N$.

Our analysis of the metastable behavior of the sequence of Markov
chain $\eta^N_t$ relies on the assumption that the set of ordered
bonds can be divided into equivalent classes in such a way that the
all jump rates in the same equivalent class are of the same order,
while the ratio between two jump rates in different classes either
vanish in the limit or tend to $+\infty$. Some terminology is
necessary to make this notion precise.

\smallskip\noindent\emph{Ordered sequences}: Consider a set of
sequences $(a^r_N: N\ge 1)$ of nonnegative real numbers indexed by
some finite set $r\in \mf R$. The set $\mf R$ is said to be
\emph{ordered} if for all $r\not = s\in\mf R$ the sequence $\arctan
\{a^r_N/ a^s_N\}$ converges as $N\uparrow\infty$.

In the examples below the set $\mf R$ will be the set of
configurations $E$ or the set of bonds $\bb B$.  Let $\bb Z_+ = \{0,
1, 2, \dots \}$, and let $\mf A_m$, $m\ge 1$, be the set of functions
$k:\bb B \to \bb Z_+$ such that $\sum_{(\eta,\xi)\in\bb B} k(\eta,\xi)
=m$.

\begin{assumption}
\label{mhyp} 
We assume that for every $m\ge 1$ the set of sequences
\begin{equation*}
\Big\{ \prod_{(\eta,\xi)\in\bb B} R_N(\eta,\xi)^{k(\eta,\xi)} : N \ge 1\Big\}
\;,\quad k\in\mf A_m
\end{equation*}
is ordered. 
\end{assumption}

We assume from now on that the sequence of Markov chains $\eta^N_t$
fulfills Assumption \ref{mhyp}. In particular, the sequences
$\{R_N(\eta,\xi) : N\ge 1\}$, $(\eta, \xi) \in \bb B$, are ordered.

\subsection{The shallowest valleys, the fastest slow variable}
\label{sec06}

We identify in this subsection the shortest time-scale at which a
metastable behavior is observed, we introduce the shallowest valleys,
and we prove that these valleys form a partition which fulfills
conditions (H1)--(H3).

We first identify the valleys. Let
\begin{equation*}
\frac 1{\alpha_N} \;=\; \sum_{\eta \in E} \sum_{\xi: \xi\not = \eta} 
R_N(\eta,\xi) \;.  
\end{equation*}
We could also have defined $\alpha^{-1}_N$ as $\max \{ R_N(\eta,\xi) :
(\eta, \xi) \in \bb B\}$.  By Assumption \ref{mhyp}, for every
$\eta\not = \xi \in E$, $\alpha_N R_N(\eta,\xi)\to R(\eta,\xi)\in
[0,1]$. Let $\lambda(\eta) = \sum_{\xi\not =\eta} R(\eta,\xi) \in \bb
R_+$, and denote by $E_0$ the subset of points of $E$ such that
$\lambda(\eta)>0$. For all $\eta\in E_0$ let
$p(\eta,\xi)=R(\eta,\xi)/\lambda(\eta)$. It is clear that for all
$\eta$, $\zeta$ in $E$, $\xi\in E_0$,
\begin{equation}
\label{01}
\lim_{N\to\infty} \alpha_N \, \lambda_N(\eta)\;=\lambda(\eta)\;,
\quad \lim_{N\to\infty} p_N(\xi,\zeta) \;=\; p(\xi,\zeta)\;.
\end{equation}

Denote by $X_R(t)$ the $E$-valued Markov chain whose jump rates
are given by $R(\eta,\xi)$. Not that this Markov chain might not be
irreducible. However, by definition of $\alpha_N$, there is at least
one bond $(\eta,\xi)\in \bb B$ such that $R(\eta,\xi)>0$.

Denote by $\ms E_1, \ms E_2, \dots, \ms E_{\mf n}$ the recurrent
classes of the Markov chain $X_R(t)$, and by $\Delta$ the set of
transient points, so that $\{\ms E_1, \dots, \ms E_{\mf n}, \Delta\}$
forms a partition of $E$:
\begin{equation}
\label{17}
E \;=\; \ms E \sqcup \Delta \;, \quad
\ms E \;=\; \ms E_1 \sqcup  \cdots \sqcup \ms E_{\mf n}
\; .
\end{equation}
Here and below we use the notation $\ms A\sqcup \ms B$ to represent
the union of two disjoint sets $\ms A$, $\ms B$: $\ms A\sqcup \ms B =
\ms A\cup \ms B$, and $\ms A\cap \ms B=\varnothing$.

Note that the sets $\ms E_x$, $x\in S= \{1, \dots, \mf n\}$, do not
depend on $N$.  If $\mf n=1$, the chain does not possess valleys. This
is the case, for instance, if the rates $R_N(x,y)$ are independent of
$N$. Assume, therefore, and up to the end of this subsection, that
$\mf n\ge 2$.

Let $\theta_N$ be defined by
\begin{equation}
\label{05}
\frac 1{\theta_N} \;=\;  \sum_{x\in S} 
\frac{\Cap_N (\ms E_x, \breve{\ms E}_x)} {\mu_N(\ms E_x)}\;. 
\end{equation}

\begin{theorem}
\label{mt1}
The partition $\ms E_1, \dots, \ms E_{\mf n}, \Delta$ and the time
scales $\alpha_N$, $\theta_N$ fulfill the conditions {\rm
  (H1)--(H3)}. Moreover, For every $x\in S$ and every $\eta\in \ms
E_x$, there exists $m_x(\eta)\in (0,1]$ such that
\begin{equation*}
\tag*{\bf (H0)}
\lim_{N\to\infty} \frac{\mu_N(\eta)}{\mu_N(\ms E_x)}\;=\; m_x(\eta)\;.
\end{equation*}
\end{theorem}

\begin{remark}
\label{rm1}
The jump rates $r_{\ms E}(x,y)$ which appear in condition (H1) are
introduced in Lemma \ref{s06}. It follows from Theorems \ref{s01} and
\ref{mt1} that in the time-scale $\theta_N$ the chain $\eta^N_t$
evolves among the sets $\ms E_x$, $x\in S$, as a Markov chain which
jumps from $x$ to $y$ at rate $r_{\ms E}(x,y)$.
\end{remark}

In the next three remarks we present some outcomes of Theorem
\ref{s01} and \ref{mt1} on the evolution of the chain $\eta^N_t$ in a
time-scale longer than $\theta_N$. These remarks anticipate the
recursive procedure of the next subsection.

\begin{remark}
\label{rm2}
The jump rates $r_{\ms E}(x,y)$ define a Markov chain on $S$,
represented by $X_{\ms E}(t)$. Denote by $T$ the set of transient
points of this chain and assume that $T\not = \varnothing$. It
follows from Theorem \ref{s01} that in the time-scale $\theta_N$,
starting from a set $\ms E_x$, $x\in T$, the chain $\eta^N_t$
leaves the set $\ms E_x$ at an asymptotically exponential time, and 
never returns to $\ms E_x$ after a finite number of visits to this
set. In particular, if we observe the chain $\eta^N_t$ in a longer
time-scale than $\theta_N$, starting from $\ms E_x$ the chain 
remains only a negligible amount of time at $\ms E_x$.
\end{remark}

\begin{remark}
\label{rm3}
Denote by $A$ the set of absorbing points of $X_{\ms E}(t)$, and
assume that $A\not = \varnothing$. In this case, in the time-scale
$\theta_N$, starting from a set $\ms E_x$, $x\in A$, the chain
$\eta^N_t$ never leaves the set $\ms E_x$. To observe a non-trivial
behavior starting from this set one has to consider longer-time
scales.
\end{remark}

\begin{remark}
\label{rm4}
Finally, denote by $\ms C_1, \dots, \ms C_{\mf p}$ the equivalent
classes of $X_{\ms E}(t)$. Suppose that there is a class, say $\ms C_1$, of
recurrent points which is not a singleton. In this case, starting from
a set $\ms E_x$, $x\in\ms C_1$, in the time-scale $\theta_N$, the
chain $\eta^N_t$ leaves the set $\ms E_x$ at an asymptotically
exponential time, and returns to $\ms E_x$ infinitely many times.

Suppose now that there are at least two classes, say $\ms C_1$ and
$\ms C_{2}$, of recurrent points. This means that in the time-scale
$\theta_N$, starting from a set $\ms E_x$, $x\in \ms C_1$, the process
never visits a set $\ms E_y$ for $y\in \ms C_2$. For this to occur one
has to observe the chain $\eta^N_t$ in a longer time-scale. 

Denote by $R_1, \dots, R_{\mf m}$ the recurrent classes of $X_{\ms
  E}(t)$.  In the next subsection, we derive a new time-scale at which
one observes jumps from sets of the form $\ms F_a = \cup_{x\in R_a}
\ms E_x$ to sets of the form $\ms F_b = \cup_{x\in R_b} \ms E_x$.
\end{remark}

\subsection{All deep valleys and  slow variables}

We obtained in the previous subsection two time-scales $\alpha_N$,
$\theta_N$, and a partition $\ms E_1, \dots, \ms E_{\mf n}, \Delta$ of
the state space $E$ which satisfy conditions (H0)--(H3).  We present
in this subsection a recursive procedure. Starting from two
time-scales $\beta^-_N$, $\beta_N$, and a partition $\ms F_1, \dots,
\ms F_{\mf p}, \Delta_{\ms F}$ of the state space $E$ satisfying the
assumptions (H0)--(H3) and such that $\mf p\ge 2$, it provides a
longer time-scale $\beta^+_N$ and a coarser partition $\ms G_1, \dots,
\ms G_{\mf q}, \Delta_{\ms G}$ which fulfills conditions (H0)--(H3)
with respect to the sequences $\beta_N$, $\beta^+_N$.

Consider a partition $\ms F_1, \dots, \ms F_{\mf p}$, $\Delta_{\ms F}$
of the set $E$ and two sequences $\beta^-_N$, $\beta_N$ such that
$\beta^-_N/\beta_N \to 0$. Assume that $\mf p\ge 2$ and that the
partition and the sequences $\beta^-_N$, $\beta_N$ satisfy conditions
(H0)--(H3). Denote by $r_{\ms F}(x,y)$ the jump rates appearing in
assumption (H1).

\smallskip\noindent\emph{The coarser partition}. Let $P=\{1, \dots,
\mf p\}$ and let $X_{\ms F}(t)$ be the $P$-valued Markov chain whose
jumps rates are given by $r_{\ms F}(x,y)$.

Denote by $G_1, G_2, \dots, G_{\mf q}$ the recurrent classes of the
chain $X_{\ms F}(t)$, and by $G_{\mf q+1}$ the set of transient
points. The sets $G_1, \dots, G_{\mf q+1}$ form a partition of $P$. We
claim that $\mf q< \mf p$. Fix $x\in P$ such that $\sum_{y\not = x}
r_{\ms F}(x,y)>0$, whose existence is guaranteed by hypothesis
(H1). Suppose that the point $x$ is transient. In this case the number
of recurrent classes must be smaller than $\mf p$. If, on the other
hand, $x$ is recurrent, the recurrent class which contains $x$ must
have at least two elements, and the number of recurrent classes must
be smaller than $\mf p$.

Let $Q=\{1, \dots, \mf q\}$, 
\begin{equation}
\label{b12}
\ms G_a \;=\; \bigcup_{x\in G_a} \ms F_x\;, \quad
\Delta_* \;=\; \bigcup_{x\in G_{\mf q+1}} \ms F_x\;, \quad
\Delta_{\ms G} = \Delta_{\ms F} \cup \Delta_*\;, \quad
a\in Q\;.
\end{equation}
Since, by \eqref{17}, $\{\ms F_1, \dots, \ms F_{\mf p}, \Delta_{\ms
  F}\}$ forms a partition of $E$, $\{\ms G_1, \dots, \ms G_{\mf q},
\Delta_{\ms G}\}$ also forms a partition of $E$:
\begin{equation}
\label{b10}
E \;=\; \ms G \sqcup \Delta_{\ms G} \;, \quad
\ms G \;=\; \ms G_1 \sqcup  \cdots \sqcup \ms G_{\mf q}
\;.
\end{equation}

\smallskip\noindent\emph {The longer time-scale}.  For $a\in Q = \{1,
\dots, \mf q\}$, let $\breve{\ms G}_a$ be the union of all leaves
except $\ms G_a$:
\begin{equation*}
\breve{\ms G}_a \;=\; \bigcup_{b\not = a} \ms G_b\;.
\end{equation*}
Assume that $\mf q>1$, and let $\beta^+_N$ be given by
\begin{equation}
\label{b05}
\frac 1{\beta^+_N} \;=\;  \sum_{a\in Q} 
\frac{\Cap_N (\ms G_a, \breve{\ms G}_a)} {\mu_N(\ms G_a)}\;. 
\end{equation}

\begin{theorem}
\label{mt2}
The partition $\ms G_1, \dots, \ms G_q$, $\Delta_{\ms G}$ and the time
scales $(\beta_N, \beta^+_N)$ satisfy conditions {\rm (H0)--(H3)}.
\end{theorem}

\begin{remark}
  It follows from Theorems \ref{s01} and \ref{mt2} that the chain
  $\eta^N_t$ exhibits a metastable behavior in the time-scale
  $\beta^+_N$ if $\mf q>1$. We refer to Remarks \ref{rm1a},
  \ref{rm1d}, \ref{rm1b} and \ref{rm1c}.
\end{remark}

\begin{remark}
\label{rm3.1}
As $\mf q<\mf p$ and as we need $\mf p$ to be greater than or equal to
$2$ to apply the iterative procedure, this recursive algorithm ends
after a finite number of steps.

If $\mf q=1$, $\beta_N$ is the longest time-scale at which a
metastable behavior is observed. In this time-scale, the chain
$\eta^N_t$ jumps among the sets $\ms F_x$ as does the chain $X_{\ms
  F}(t)$ until it reaches the set $\ms G_1 = \cup_{x\in G_1} \ms
F_x$. Once in this set, it remains there for ever jumping
among the sets $\ms F_x$, $x\in G_1$, as the Markov chain $X_{\ms
  F}(t)$, which restricted to $G_1$ is an irreducible Markov chain.
\end{remark}

\noindent{\bf The successive valleys}: Observe that the valleys $\ms
G_a$ were obtained as the recurrent classes of the Markov chain
$X_{\ms F}(t)$: $\ms G_a = \cup_{x\in G_a} \ms F_x$, where $G_a$ is a
recurrent class of $X_{\ms F}(t)$. In particular, at any time-scale
the valleys are formed by unions of the valleys obtained in the first
step of the recursive argument, which were denoted by $\ms E_x$ in the
previous subsection. Moreover, by (H0), each configuration in $\ms
G_a$ has measure of the same order.

\smallskip\noindent{\bf Conclusion}: We presented an iterative method
which provides a finite sequence of time-scales and of partitions of
the set $E$ satisfying conditions (H0)-(H3). At each step, the time
scales become longer and the partitions coarser. By Theorem \ref{s01},
to each pair of time-scale and partition corresponds a metastable
behavior of the chain $\eta^N_t$. This recursive algorithm provides
all time-scales at which a metastable behavior of the chain $\eta^N_t$
is observed, and all slow variables which keep a Markovian dynamics.

\section{What do we learn from Assumption \ref{mhyp}?}
\label{sec02}

We prove in this section that the jump rates of the trace processes
satisfy Assumption \ref{mhyp}, and that some sequences, such as the
one formed by the measures of the configurations, are ordered.

\begin{asser}
\label{as01}
Let $F$ be a proper subset of $E$ and denote by $R^F_N(\eta,\xi)$, $\eta\not =
\xi\in F$, the jump rates of the trace of $\eta^N_t$ on $F$. The jump
rates $R^F_N(\eta,\xi)$ satisfy Assumption \ref{mhyp}.
\end{asser}

\begin{proof}
We prove this assertion by removing one by one the elements of
$E\setminus F$. Assume that $F=E\setminus \{\zeta\}$ for some $\zeta\in E$. By
Corollary 6.2 in \cite{bl2} and by the equation following the proof of
this corollary, for $\eta\not = \xi\in F$, $R^F_N(\eta,\xi) =
R_N(\eta,\xi) + R_N(\eta,\zeta) p_N(\zeta,\xi)$. Hence, 
\begin{equation}
\label{18}
R^F_N(\eta,\xi) \;=\; \frac{ \sum_{w\in E} R_N(\eta,\xi)R_N(\zeta,w) + R_N(\eta,\zeta)
  R_N(\zeta,\xi)} {\sum_{w\in E} R_N(\zeta,w)}\;\cdot
\end{equation}
It is easy to check from this identity that Assumption \ref{mhyp}
holds for the jump rates $R^F_N$. It remains to proceed recursively to
complete the proof.
\end{proof}

\begin{lemma}
\label{as02}
The sequences $\{\mu_N(\eta) : N\ge 1\}$, $\eta\in E$, are ordered.
\end{lemma}

\begin{proof}
Fix $\eta\not = \xi\in E$ and let $F=\{\eta,\xi\}$. By
\cite[Proposition 6.3]{bl2}, the stationary state of the trace of
$\eta^N_t$ on $F$, denoted by $\mu^F_N$, is given by $\mu^F_N(\eta)=
\mu_N(\eta)/\mu_N(F)$. As $\mu^F_N$ is the invariant probability
measure, $\mu^F_N(\eta) R^F_N(\eta,\xi) = \mu^F_N(\xi)
R^F_N(\xi,\eta)$. Therefore, $\mu_N(\eta)/\mu_N(\xi) =
\mu^F_N(\eta)/\mu^F_N(\xi) = R^F_N(\xi,\eta)/R^F_N(\eta,\xi)$.  By
Assertion \ref{as01}, the sequences $\{R^F_N(a,b) : N\ge 1\}$, $a\not
= b\in \{\eta,\xi\}$ are ordered. This completes the proof of the
lemma.
\end{proof}

The previous lemma permits to divide the configurations of $E$ into
equivalent classes by declaring $\eta$ equivalent to $\eta'$, $\eta
\sim \eta'$, if $\mu_N(\eta)/\mu_N(\eta')$ converges to a real number
belonging to $(0,\infty)$.

\begin{asser}
\label{as16}
Let $F$ be a proper subset of $E$. For every bond $(\eta',\xi')\in \bb
B$ and every $m\ge 1$ the set of sequences
\begin{equation*}
\Big\{ \prod_{(\eta,\xi)\in\bb B} R^F_N(\eta,\xi)^{k(\eta,\xi)} R_N(\eta',\xi') : N \ge 1\Big\}
\;,\quad k\in\mf A_m
\end{equation*}
is ordered. 
\end{asser}

\begin{proof}
We proceed as in the proof of Assertion \ref{as01}, by removing one by
one the elements of $E\setminus F$. Fix $\zeta\in E\setminus F$.
It follows from \eqref{18} and from Assumption \ref{mhyp} that the
claim of the assertion holds for $F'= E\setminus \{\zeta\}$.

Fix $\zeta'\in E\setminus F$, $\zeta'\not = \zeta$.  By using formula
\eqref{18}, to express the rates $R^{E \setminus \{\zeta, \zeta'\}}$
in terms of the rates $R^{E \setminus \{\zeta\}}$, and the statement
of this assertion for $F'= E\setminus \{\zeta\}$ we prove that this
assertion also holds for $F'= E\setminus \{\zeta, \zeta'\}$. Iterating
this algorithm we complete the proof of the assertion.
\end{proof}

Denote by $c_N(\eta,\xi) = \mu_N(\eta) R_N(\eta,\xi)$, $(\eta,\xi)\in
\bb B$, the (generally asymmetric) conductances.

\begin{lemma}
\label{s11}
The conductances $\{c_N(\eta,\xi) : N\ge 1\}$, $(\eta,\xi) \in \bb B$,
are ordered. 
\end{lemma}

\begin{proof}
Consider two bonds $(\eta,\xi)$, $(\eta',\xi')$ in $\bb B$. As in the
proof of Lemma \ref{as02}, we may express the ratio of the
conductances as
\begin{equation*}
\frac{c_N(\eta,\xi)}{c_N(\eta',\xi')}\;=\;
\frac{\mu_N (\eta) R_N(\eta,\xi)}{\mu_N (\eta') R_N(\eta',\xi')}\;=\;
\frac{R^F_N (\eta',\eta) R_N(\eta,\xi)}{R^F_N (\eta,\eta')
  R_N(\eta',\xi')}\;, 
\end{equation*}
where $F=\{\eta,\eta')$. It remains to recall the statement of
assertion \ref{as16} to complete the proof of the lemma.
\end{proof}

Denote by $\bb B^s$ the symmetrization of the set $\bb B$, that is,
the set of bonds $(\eta,\xi)$ such that $(\eta,\xi)$ or $(\xi,\eta)$
belongs to $\bb B$:
\begin{equation*}
\bb B^s \;=\; \big \{(\eta,\xi) \in E\times E : \eta\not = \xi \,,\,
(\eta,\xi)\in \bb B \;\text{ or }\; (\xi, \eta)\in \bb B \big\}\;.
\end{equation*}
Denote by $c^s_N(\eta,\xi)$, $(\eta,\xi)\in \bb B^s$, the symmetric
part of the conductance:
\begin{equation}
\label{23}
c^s_N(\eta,\xi) \;=\; \frac 12 \big\{c_N(\eta,\xi) + c_N(\xi, \eta) \big\}\;.
\end{equation}
Next result is a straightforward consequence of the previous lemma. 

\begin{corollary}
\label{s12}
The symmetric conductances $\{c^s_N(\eta,\xi) : N\ge 1\}$, $(\eta,\xi) \in \bb B^s$,
are ordered. 
\end{corollary}

As in Lemma \ref{as02}, the previous corollary permits to divide the
set $\bb B^s$ into equivalent classes by declaring $(\eta,\xi)$
equivalent to $(\eta',\xi')$, $(\eta,\xi) \sim (\eta',\xi')$, if
$c^s_N(\eta,\xi)/c^s_N(\eta',\xi')$ converges to a constant in
$(0,\infty)$.

It is possible to deduce from Assumption \ref{mhyp} that many other
sequences are ordered. We do not present these results here as we do
not use them below.

\section {Cycles, sector condition and capacities} 
\label{sec00}

We prove in this section that the generator of a Markov chain on a
finite set can be decomposed as the sum of cycle generators and that
it satisfies a sector condition. This last bound permits to estimate
the capacity between two sets by the capacity between the same sets
for the reversible process.

Throughout this section, $E$ is a fixed finite set and $\mc L$
represents the generator of an $E$-valued, continuous-time Markov
chain.  We adopt all notation introduced in Section \ref{not}, removing
the index $N$ since the chain is fixed.  We start with some
definitions.

In a finite set, the decomposition of a generator into cycle
generators is very simple. The problem for infinite sets is much more
delicate. We refer to \cite{gv12} for a discussion of the question.

\smallskip\noindent\emph{Cycle}: A cycle is a sequence of distinct
configurations $(\eta_0, \eta_1, \dots, \eta_{n-1}, \eta_n=\eta_0)$
whose initial and final configuration coincide: $\eta_i \not =
\eta_j\in E$, $i\not = j\in \{0, \dots, n-1\}$. The number $n$ is
called the length of the cycle.

\smallskip\noindent\emph{Cycle generator}: A generator $\mc L$ of an
$E$-valued Markov chain, whose jump rates are denoted by
$R(\eta,\xi)$, is said to be a cycle generator associated to the cycle
$\mf c = (\eta_0, \eta_1, \dots, \eta_{n-1}, \eta_n=\eta_0)$ if there
exists reals $r_i>0$, $0\le i<n$, such that
\begin{equation*}
R(\eta,\xi) \;=\; 
\begin{cases}
r_i & \text{if $\eta=\eta_i$ and $\xi=\eta_{i+1}$ for some $0\le i
  <n$}\;, \\
0 & \text{otherwise}\;.
\end{cases}
\end{equation*}
We denote this cycle generator by $\mc L_{\mf c}$. Note that
\begin{equation*}
(\mc L_{\mf c} f)(\eta) \;=\; \sum_{i=0}^{n-1} \mb 1\{\eta =
\eta_i\}\, r_i \, [f(\eta_{i+1}) - f(\eta_i)]\;.
\end{equation*}

\smallskip\noindent\emph{Sector condition}: A generator $\mc L$ of an
$E$-valued, irreducible Markov chain, whose unique invariant
probability measure is denoted by $\mu$, is said to satisfy a sector
condition if there exists a constant $C_0<\infty$ such that for all
functions $f$, $g:E\to \bb R$,
\begin{equation*}
\< \mc Lf, g\>^2_\mu \;\le\; C_0 \< (-\mc L f), f\>_\mu\, \< (-\mc L g), g\>_\mu\;. 
\end{equation*}
In this formula, $\< f, g\>_\mu$ represents the scalar product in
$L^2(\mu)$: 
\begin{equation*}
\< f, g\>_\mu \;=\; \sum_{\eta\in E} f(\eta)\, g(\eta) \,
\mu(\eta)\;. 
\end{equation*}

We claim that every cycle generator satisfies a sector condition and
that every generator $\mc L$ of an $E$-valued Markov chain, stationary with
respect to a probability measure $\mu$, can be decomposed as the sum
of cycle generators which are stationary with respect to $\mu$.

\begin{asser}
\label{as14}
Consider a cycle $\mf c = (\eta_0, \eta_1, \dots, \eta_{n-1},
\eta_n=\eta_0)$ of length $n\ge 2$ and let $\mc L$ be a cycle generator
associated to $\mf c$. Denote the jump rates of $\mc L$ by
$R(\eta_i,\eta_{i+1})$. A measure $\mu$ is stationary for $\mc L$ if and
only if 
\begin{equation}
\label{03}
\mu(\eta_i) \, R(\eta_i,\eta_{i+1}) \;\; \text{is constant}\;.  
\end{equation}
\end{asser}

The proof of the previous assertion is elementary and left to the
reader. The proof of the next one can be found in \cite[Lemma
5.5.8]{klo12}.

\begin{asser}
\label{as15}
Let $\mc L$ be a cycle generator associated to a cycle $\mf c$ of length
$n$. Then, $\mc L$ satisfies a sector condition with constant $2n$: For all
$f$, $g:E\to \bb R$,
\begin{equation*}
\< \mc Lf, g\>^2_\mu \;\le\; 2n \, \< (-\mc L f), f\>_\mu\, \< (-\mc L g), g\>_\mu\;. 
\end{equation*}
\end{asser}

\begin{lemma}
\label{s03}
Let $\mc L$ be a generator of an $E$-valued, irreducible Markov chain.
Denote by $\mu$ the unique invariant probability measure.  Then, there
exists cycles $\mf c_1, \dots, \mf c_p$ such that
\begin{equation*}
\mc L \;=\; \sum_{j=1}^p \mc L_{\mf c_j}\;,
\end{equation*}
where $\mc L_{\mf c_j}$ are cycle generators associated to $\mf c_j$ which
are stationary with respect to $\mu$.
\end{lemma}

\begin{proof}
The proof consists in eliminating successively all $2$-cycles (cycles
of length $2$), then all $3$-cycles and so on up to the $|E|$-cycle if
there is one left. Denote by $R(\eta,\xi)$ the jump rates of the
generator $\mc L$ and by $\bb C_2$ the set of all $2$-cycles
$(\eta,\xi, \eta)$ such that $R(\eta,\xi) R(\xi,\eta)>0$. Note that
the cycle $(\eta,\xi,\eta)$ coincide with the cycle $(\xi,\eta,\xi)$.

Fix a cycle $\mf c = (\eta,\xi,\eta)\in\bb C_2$. Let $\bar c(\eta,\xi) =
\min\{ \mu(\eta) R(\eta,\xi) , \mu(\xi) R(\xi,\eta)\}$ be the minimal
conductance of the edge $(\eta,\xi)$, and let $R_{\mf c} (\eta,\xi)$
be the jump rates given by $R_{\mf c} (\eta,\xi) =
\bar c(\eta,\xi)/\mu(\eta)$, $R_{\mf c} (\xi,\eta) =
\bar c(\eta,\xi)/\mu(\xi)$. Observe that $R_{\mf c} (\zeta,\zeta') \le R
(\zeta,\zeta')$ for all $(\zeta,\zeta')$, and that $R_{\mf c} (\xi,\eta) = R
(\xi,\eta)$ or $R_{\mf c} (\eta,\xi) = R (\eta,\xi)$.

Denote by $\mc L_{\mf c}$ the generator associated the the jump rates
$R_{\mf c}$. Since $\mu(\eta) R_{\mf c} (\eta,\xi) = \bar c(\eta,\xi) =
\mu(\xi) R_{\mf c} (\xi,\eta)$, by \eqref{03}, $\mu$ is a stationary
state for $\mc L_{\mf c}$ (actually, reversble). Let $\mc L_1 = \mc L - \mc L_{\mf
  c}$ so that
\begin{equation*}
\mc L \;=\; \mc L_1 \;+\; \mc L_{\mf c}\;.
\end{equation*}
As $R_{\mf c} (\zeta,\zeta') \le R (\zeta,\zeta')$, $\mc L_1$ is the
generator of a Markov chain. Since both $\mc L$ and $\mc L_{\mf c}$ are
stationary for $\mu$, so is $\mc L_1$. Finally, if we draw an arrow
from $\zeta$ to $\zeta'$ if the jump rate from $\zeta$ to $\zeta'$ is
strictly positive, the number of arrows for the generator $\mc L_1$ is
equal to the number of arrows for the generator $\mc L$ minus $1$ or $2$.
This procedure has therefore strictly decreased the number of arrows
of $\mc L$.

We may repeat the previous algorithm to $\mc L_1$ to remove from $\mc L$ all
$2$-cycles $(\eta,\xi, \eta)$ such that $R(\eta,\xi) R(\xi,\eta)>0$.
Once this has been accomplished, we may remove all $3$-cycles
$(\eta_0,\eta_1, \eta_2, \eta_3=\eta_0)$ such that $\prod_{0\le i <3}
R(\eta_i,\eta_{i+1}) >0$. At each step at least one arrow is removed
from the generator which implies that after a finite number of steps
all $3$-cycles are removed.

Once all $k$-cycles have been removed, $2\le k<|E|$, we have obtained
a decomposition of $\mc L$ as
\begin{equation*}
\mc L \;=\; \sum_{k=2}^{|E|-1} \mc L_k \;+\; \hat {\mc L}\;,
\end{equation*}
where $\mc L_k$ is the sum of $k$-cycle generators and is stationary
with respect to $\mu$, and $\hat {\mc L}$ is a generator, stationary with
respect to $\mu$, and with no $k$-cycles, $2\le k<|E|$. If $\hat {\mc L}$ has
an arrow, as it is stationary with respect to $\mu$ and has no
$k$-cycles, $\hat {\mc L}$ must be an $|E|$-cycle generator, providing the
decomposition stated in the lemma.
\end{proof}

\begin{remark}
\label{s08}
Observe that a generator $\mc L$ is reversible with respect to $\mu$ if
and only if it has a decomposition in $2$-cycles. Given a measure
$\mu$ on a finite state space, for example the Gibbs measure
associated to a Hamiltonian at a fixed temperature, by introducing
$k$-cycles satisfying \eqref{03} it is possible to define
non-reversible dynamics which are stationary with respect to
$\mu$. The previous lemma asserts that this is the only way to define
such dynamics.
\end{remark}

\begin{corollary}
\label{s09}
The generator $\mc L$ satisfies a sector condition with constant bounded
by $2|E|$: For all
$f$, $g:E\to \bb R$,
\begin{equation*}
\< \mc Lf, g\>^2_\mu \;\le\; 2|E| \, \< (-\mc L f), f\>_\mu\, \< (-\mc L g), g\>_\mu\;. 
\end{equation*} 
\end{corollary}

\begin{proof}
Fix $f$ and $g:E\to \bb R$. By Lemma \ref{s03},
\begin{equation*}
\< \mc Lf, g\>^2_\mu \;=\; \Big( \sum_{j=1}^p \< \mc L_{\mf c_j} f, g\>_\mu
\Big)^2  \;,
\end{equation*}
where $\mc L_{\mf c_j}$ is a cycle generator, stationary with respect to
$\mu$, associated to the cycle $\mf c_j$. By Assertion \ref{as15} and by
Schwarz inequality, since all cycles have length at most $|E|$, the
previous sum is bounded by
\begin{equation*}
2|E| \, \sum_{j=1}^p \< (-\mc L_{\mf c_j} f), f\>_\mu 
\, \sum_{k=1}^p \< (-\mc L_{\mf c_k} g), g\>_\mu \;=\; 
2|E| \, \< (-\mc L f), f\>_\mu \, \< (-\mc L g), g\>_\mu\;,
\end{equation*}
as claimed
\end{proof}

Denote by $R^s(\eta,\xi)$ the symmetric part of the jump rates
$R^s(\eta,\xi)$:
\begin{equation}
\label{30}
R^s(\eta,\xi) \;=\; \frac 12 \Big\{ R(\eta,\xi) \;+\; 
\frac{\mu(\xi)}{\mu (\eta)} \, R (\xi,\eta) \Big\}\;.
\end{equation}
Denote by $\eta^s_t$ the $E$-valued Markov chain whose jump rates are
given by $R^s$. The chain $\eta^s_t$ is called the reversible chain.

For two disjoint subsets $A$, $B$ of $E$, denote by $\Cap(A,B)$
(resp. $\Cap^s(A,B)$) the capacity between $A$ and $B$ (for the
reversible chain). Next result follows from Corollary \ref{s09} and
Lemmas 2.5 and 2.6 in \cite{gl14}

\begin{corollary}
\label{s10}
Fix two disjoint subsets $A$, $B$ of $E$. Then,
\begin{equation*}
\Cap^s(A,B) \;\le\; \Cap(A,B) \;\le\; 2|E| \, \Cap^s(A,B)\;.
\end{equation*}
\end{corollary}

We conclude the section with an identity and an inequality which will
be used several times in this article.  Let $A$ and $B$ be two
disjoint subsets of $E$. By definition of the capacity
\begin{equation*}
\Cap (A, B)
\; =\; \sum_{\eta\in A} 
\mu(\eta)\, \lambda(\eta)\, \bb P_\eta \big[ H_{B} < H^+_{A} \big]
\; =\; \sum_{\eta\in A} 
\mu(\eta)\, \lambda(\eta)\, \sum_{\xi\in B}
\bb P_\eta \big[ H_{\xi} = H^+_{A\cup B} \big] \;.  
\end{equation*} 
Therefore, if we denote by $R^{A\cup B} (\eta,\xi)$, $\eta\not
=\xi\in A\cup B$, the jump rates of the trace of the chain $\eta_t$
on the set $A\cup B$, by \cite[Proposition 6.1]{bl2},
\begin{equation}
\label{12}
\Cap (A,B) \;=\; \sum_{\eta\in A} \mu(\eta) \sum_{\xi\in B}
R^{A\cup B} (\eta, \xi) \;.  
\end{equation} 

Let $A$ be a non-empty subset of $E$ and denote by $R^A(\eta,\xi)$ the
jump rates of the trace of $\eta_t$ on $A$.  We claim that for all
$\eta \not = \xi\in A$,
\begin{equation}
\label{19}
\mu(\eta)\, R^A(\eta,\xi) \;\le\; \Cap (\eta,\xi) \;.  
\end{equation}

Denote by $\lambda^A(\zeta)$ the holding rates of the trace process on
$A$ and by $p^A(\zeta,\zeta')$ the jump probabilities.  By definition,
\begin{equation*}
R^A(\eta,\xi) \;=\; \lambda^A(\eta)\, p^A(\eta,\xi)
\;=\; \lambda^A(\eta)\, \bb P_\eta[H_\xi = H^+_{A}] \;\le\;
\lambda^A(\eta)\, \bb P_\eta[H_\xi < H^+_{\eta}]\;.
\end{equation*}
Multiplying both sides of this inequality by
$\mu_A(\eta)=\mu(\eta)/\mu(A)$, by definition of the capacity we
obtain that
\begin{equation*}
\mu_A(\eta) \, R^A(\eta,\xi) \;\le \; \Cap_A(\eta,\xi)\;,
\end{equation*}
where $\Cap_A(\eta,\xi)$ stands for the capacity with respect to the
trace process on $A$. To complete the proof of \eqref{19}, it remains
to recall formula (A.10) in \cite{bl7}.

\section{Reversible chains  and capacities}
\label{sec03}

We present in this section some estimates for the capacity of
reversible, finite state Markov chains obtained in in
\cite{bl4}. There are useful below since se proved in Corollary
\ref{s10} that the capacity between two disjoint subsets $\ms A$, $\ms
B$ of $E$ is of the same order as the capacity with respect to the
reversible chain.

Recall from \eqref{23} that we denote by $c^s_N(\eta,\xi)$ the
symmetric conductance of the bond $(\eta,\xi)$.  Fix two disjoint
subsets $\ms A$, $\ms B$ of $E$.  A self-avoiding path $\gamma$ from
$\ms A$ to $\ms B$ is a sequence of configurations $(\eta_0, \eta_1,
\dots, \eta_n)$ such that $\eta_0\in \ms A$, $\eta_n\in \ms B$,
$\eta_i \not = \eta_j$, $i\not =j$, $c^s_N(\eta_i,\eta_{i+1})>0$,
$0\le i <n$. Denote by $\Gamma_{\ms A,\ms B}$ the set of self-avoiding
paths from $\ms A$ to $\ms B$ and let
\begin{equation}
\label{29}
\bs c^s_N(\gamma) \; =\; \min_{0\le i<n} c^s_N(\eta_i,\eta_{i+1})
\;, \quad
\bs c^s_N(\ms A,\ms B) \; =\; \max_{\gamma\in \Gamma_{\ms A,\ms B}} \bs c^s_N(\gamma) \;.
\end{equation}
For two configurations $\eta$, $\xi$, we represent $\bs
c^s_N(\{\eta\},\{\xi\})$ by $\bs c^s_N(\eta,\xi)$.  Note that $\bs
c^s_N(\eta,\xi) \le c^s_N(\eta,\xi) $, with possibly a strict inequality.

Fix two disjoint subsets $\ms A$, $\ms B$ of $E$ and a configuration
$\eta\not\in \ms A\cup \ms B$. We claim that
\begin{equation}
\label{25}
\bs c^s_N(\ms A,\ms B) \;\ge\;  \min \{ \bs c^s_N(\ms A,\eta) \,,\, \bs c^s_N(\eta,\ms B) \}\;.
\end{equation}
Indeed, there exist a self-avoiding path $\gamma_1$ from $\ms A$ to
$\eta$, and a self-avoiding path $\gamma_2$ from $\eta$ to $\ms B$ such
that $\bs c^s_N(\ms A,\eta) = \bs c^s_N(\gamma_1)$, $\bs c^s_N(\eta,\ms B) = \bs
c^s_N(\gamma_2)$. Juxtaposing the paths $\gamma_1$ and $\gamma_2$, we
obtain a path $\gamma$ from $\ms A$ to $\ms B$. Of course, the path $\gamma$
may not be self-avoiding, may return to $\ms A$ before reaching $\ms B$, or
may reach $\ms B$ before hitting $\eta$. In any case, we may obtain from
$\gamma$ a subpath $\hat\gamma$ which is self-avoiding and which
connects $\ms A$ to $\ms B$. Subpath in the sense that all bonds
$(\eta_i,\eta_{i+1})$ which appear in $\hat \gamma$ also appear in
$\gamma$. In particular, 
\begin{equation*}
\bs c^s_N(\hat \gamma) \;\ge\; \bs c^s_N(\gamma) \;=\; 
\min\{\bs c^s_N(\gamma_1) \,,\, \bs c^s_N(\gamma_2)\}
\;=\; \min \{ \bs c^s_N(\ms A,\eta) \,,\, \bs c^s_N(\eta,\ms B) \}
\;.
\end{equation*}
To complete the proof of claim \eqref{25}, it remains to observe that
$\bs c^s_N(\ms A,\ms B) \ge \bs c^s_N(\hat\gamma)$.

Fix two disjoint subsets $\ms A$, $\ms B$ of $E$ and configurations
$\eta_i\not\in \ms A\cup \ms B$, $1\le i\le n$, such that $\eta_i \not =
\eta_j$, $i\not = j$. Iterating inequality \eqref{25} we obtain that 
\begin{equation}
\label{26}
\bs c^s_N(\ms A,\ms B) \;\ge\;  \min \{ \bs c^s_N(\ms A,\eta_1) \,,\, 
\bs c^s_N(\eta_1,\eta_2) \,,\, \dots \,,\,  \bs c^s_N(\eta_{n-1},\eta_n) \,,\,
\bs c^s_N(\eta_n,\ms B) \}\;.  
\end{equation}

We conclude this section relating the symmetric capacity between two
sets $\ms A$, $\ms B$ of $E$ to the symmetric conductances $\bs
c^s_N(\ms A,\ms B)$. By Corollary \ref{s12}, the sequences of
symmetric conductances $\{c^s_N(\eta,\xi):N\ge 1\}$, $(\eta,\xi)\in\bb
B^s$, are ordered.  It follows from this fact and from the proof of
Lemmas 4.1 in \cite{bl4} that there exists constants
$0<c_0<C_0<\infty$ such that
\begin{equation}
\label{24}
c_0 \;<\; \liminf_{N\to\infty} \frac{\Cap^s_N(\ms A,\ms B)}
{\bs c^s_N(\ms A,\ms B)}\; \le\; \limsup_{N\to\infty} \frac{\Cap^s_N(\ms A,\ms B)}
{\bs c^s_N(\ms A,\ms B)}\;\le\; C_0\;. 
\end{equation}

\section{Proof of Theorem \ref{s01}}
\label{sec05}

In view of Theorem 5.1 in \cite{l-soft}, Theorem \ref{s01} follows
from from condition (H3) and from Propositions \ref{s04} below. Denote
by $\psi_{\ms E} : \ms E\to \{1, \dots, \mf n\}$ the projection
defined by $\psi_{\ms E}(\eta) = x$ if $\eta\in \ms E_x$:
\begin{equation*}
\psi_{\ms E}(\eta) \;=\; \sum_{x\in S} x \, \mb 1\{\eta \in \ms E_x\} \;.
\end{equation*}

\begin{proposition}
\label{s04}
Fix $x\in S$ and a configuration $\eta\in\ms E_x$. Starting from
$\eta$, the speeded-up, hidden Markov chain $X_N(t) = \psi_{\ms
  E}\big(\eta^{\ms E}(\theta_N t)\big)$ converges in the Skorohod
topology to the continuous-time Markov chain $X_{\ms E}(t)$, introduced
in Theorem \ref{s01}, which starts from $x$.
\end{proposition}

\begin{lemma}
\label{s02}
For every $x\in S$ for which $\ms E_x$ is not a singleton and for all
$\eta\not =\xi\in \ms E_x$,
\begin{equation*}
\lim_{N\to \infty} \frac{\Cap_N(\ms E_x,\breve{\ms E}_x)}
{\Cap_N(\eta,\xi)}\;=\;0\; .
\end{equation*}
\end{lemma}

\begin{proof}
Fix $x\in S$.  By \eqref{12}, applied to $A= \ms E_x$, $B=\breve{\ms
  E}_x$, and by assumption (H1),
\begin{equation*}
\lim_{N\to\infty} \theta_N \, 
\frac{\Cap_N (\ms E_x, \breve{\ms E}_x)}
{\mu_N(\ms E_x)}\;=\; \sum_{y\not =x} r_{\ms E}(x,y) \;\in\;\bb R_+
\;.
\end{equation*}
The claim of the lemma follows from this equation, from assumption
(H2) and from the fact that $\alpha_N/\theta_N\to 0$.
\end{proof}

\begin{proof}[Proof of Proposition \ref{s04}]
In view of Theorem 2.1 in \cite{bl7}, the claim of the proposition
follows from condition (H1), and from Lemma \ref{s02}.
\end{proof}

\section{Proof of Theorem \ref{mt1}}
\label{sec01}

The proof of Theorem \ref{mt1} is divided in several steps.

\smallskip\noindent{\bf 1. The measure of the metastable sets.}  We
start proving that condition (H0) is in force. Recall from Section
\ref{not} that we denote by $X_R(t)$ the $E$-valued chain which jumps
from $\eta$ to $\xi$ at rate $R(\eta,\xi)$.  Denote by $\ms C_1,
\dots, \ms C_{\mf m}$ the equivalent classes of the chain $X_R(t)$.

\begin{asser}
\label{as03}
For all $1\le j\le \mf m$, and for all $\eta\not = \xi\in \ms C_j$, there
exists $m(\eta,\xi)\in (0,\infty)$ such that
\begin{equation*}
\lim_{N\to\infty} \frac{\mu_N(\eta)}{\mu_N(\xi)} \; = \; m(\eta,\xi)\;.
\end{equation*}
\end{asser}

\begin{proof}
Fix $1\le j\le \mf m$ and $\eta\not = \xi\in \ms C_j$. By assumption,
there exists a path $(\eta=\eta_0, \dots, \eta_n=\xi)$ such that
$R(\eta_i,\eta_{i+1})>0$ for $0\le i<n$. On the other hand, since $\mu_N$ is
an invariant probability measure,
\begin{align*}
\lambda_N(\xi) \, \mu_N(\xi) \; &=\; \sum_{\zeta_0, \zeta_1, \dots, \zeta_{n-1}\in E} 
\mu_N(\zeta_0) \, \lambda_N(\zeta_0) \, p_N(\zeta_0,\zeta_1) \cdots 
p_N(\zeta_{n-1}, \xi) \\
&\ge \; \mu_N(\eta_0) \, \lambda_N(\eta_0) \, p_N(\eta_0,\eta_1) \cdots 
p_N(\eta_{n-1}, \xi)\;.
\end{align*}
Therefore,
\begin{equation*}
\frac{\mu_N(\xi)}{\mu_N(\eta)} \;\ge\; 
\frac{\lambda_N(\eta)}{\lambda_N(\xi)} \, 
p_N(\eta,\eta_1) \cdots  p_N(\eta_{n-1}, \xi) \;.  
\end{equation*}
Since $R(\eta_i,\eta_{i+1})>0$ for $0\le i<n$, by \eqref{01},
$p_N(\eta_{i}, \eta_{i+1})$ converges to $p(\eta_{i},
\eta_{i+1})>0$. For the same reason, $\alpha_N \lambda_N(\eta)$
converges to $\lambda(\eta)\in (0,\infty)$. Finally, as $\xi$ and
$\eta$ belong to the same equivalent class, there exists a path from
$\xi$ to $\eta$ with similar properties to the one from $\eta$ to
$\xi$, so that $\alpha_N \lambda_N(\xi)$ converges to $\lambda(\xi)\in
(0,\infty)$.  In conclusion,
\begin{equation*}
\liminf_{N\to\infty} \frac{\mu_N(\xi)}{\mu_N(\eta)} \;>\; 0\;.
\end{equation*}
Replacing $\eta$ by $\xi$ we obtain that $\liminf
\mu_N(\eta)/\mu_N(\xi)>0$. Since by Lemma \ref{as02} the sequences
$\{\mu_N(\zeta) : N\ge 1\}$, $\zeta\in E$, are ordered,
$\mu_N(\eta)/\mu_N(\xi)$ must converge to some value in $(0,\infty)$.
\end{proof}

By the previous assertion for every $x\in S$ and $\eta\in\ms E_x$,
\begin{equation}
\label{09}
m_x(\eta) \;:=\; \lim_{N\to\infty}  \frac{\mu_N(\eta)}{\mu_N(\ms E_x)} \;\in\;
(0,1]\;,
\end{equation}
where we adopted the convention established in condition (H1) of
Section \ref{not}.

\smallskip\noindent{\bf 2. The time-scale.}  In this subsection, we
introduce a time-scale $\gamma_N$, we prove that it is much longer
than $\alpha_N$ and that it is of the same order of $\theta_N$. In
particular the requirement $\alpha_N/\theta_N\to 0$ is in force.

Denote by $\{\eta^{\ms E}_t : t\ge 0\}$ the trace of $\eta^N_t$ on the
set $\ms E$, and by $R^{\ms E}_N : \ms E \times \ms E\to \bb R_+$ the
jump rates of $\eta^{\ms E}_t$. Let
\begin{equation}
\label{02}
\frac 1{\gamma_N} \;=\; \sum_{x\in S} 
\sum_{\eta\in\ms E_x} \sum_{\xi \in \breve{\ms E}_x} R^{\ms E}_N(\eta,\xi) \;,
\end{equation}
where $\breve{\ms E}_x$ has been introduced in \eqref{27}.  The
sequence $\gamma_N$ represents the time needed to reach the set
$\breve{\ms E}_x $ starting from $\ms E_x$ for some $x\in S$. This
time scale might be longer for other sets $\ms E_y$, $y\not =x$, but
it is of the order $\gamma_N$ at least for one $x\in S$.  We could as
well have defined $\gamma_N$ as $\max_{x\in S} \max_{\eta\in\ms E_x}$
$\max_{\xi \in \breve{\ms E}_x} R^{\ms E}_N(\eta,\xi)$.

\begin{asser}
\label{as06}
The time scale $\gamma_N$ is much longer than the time-scale
$\alpha_N$:
\begin{equation*}
\lim_{N\to\infty} \frac{\alpha_N}{\gamma_N}\;=\;0\;.
\end{equation*}
\end{asser}

\begin{proof}
We have to show that $\alpha_N  R^{\ms E}_N (\eta,\xi)$ converges to
$0$ as $N\uparrow\infty$, for all $\eta\in\ms E_x$, $\xi\in\ms E_y$,
$x\not =y\in S$. Fix $x\not =y\in S$, $\eta\in\ms E_x$, $\xi\in\ms
E_y$. Since $\ms E_x$ is a recurrent class, $R(\eta,\zeta)=0$ for all
$\zeta\not\in \ms E_x$. On the other hand, by \cite[Proposition
6.1]{bl2} and by the strong Markov property, 
\begin{equation*}
R^{\ms E}_N (\eta,\xi) \;=\; \lambda_N(\eta)\, \bb P_\eta [ H_\xi =
H^+_{\ms E}] \;=\; R_N (\eta,\xi) \;+\; \sum_{\zeta\not\in \ms E} 
R_N (\eta,\zeta) \, \bb P_\zeta [ H_\xi = H_{\ms E}]\;. 
\end{equation*}
Since $R(\eta,\zeta)=0$ for all $\zeta\not\in \ms E_x$, it follows
from the previous identity and from the definition of $R(\eta,\zeta)$
that $\alpha_N R^{\ms E}_N (\eta,\xi) \to 0$, as claimed.
\end{proof}

By Assertion \ref{as01}, for all $x\in S$, $\eta\in \ms E_x$, $\xi\in
\breve{\ms E}_x$, with the convention adopted in condition (H1) of
Section \ref{not},
\begin{equation}
\label{06}
r_{\ms E} (\eta,\xi) \;:=\; \lim_{N\to\infty} 
\gamma_N\, R^{\ms E}_N (\eta, \xi) \;\in\; [0,1]\;.
\end{equation}

\begin{asser}
\label{as04}
For all $x\in S$, 
\begin{equation*}
\ell_x \;:=\; \lim_{N\to\infty} \gamma_N \, 
\frac{\Cap_N (\ms E_x, \breve{\ms E}_x)}
{\mu_N(\ms E_x)}\;\in\; \bb R_+ \;. \quad\text{Moreover}\;,\;\;
\ell \;=\; \sum_{x\in S} \ell_x \;>\;0 \;.
\end{equation*}
\end{asser}

\begin{proof}
By \eqref{12}, applied to $A= \ms E_x$, $B=\breve{\ms E}_x$, by \eqref{09}
and by \eqref{06},  
\begin{equation*}
\lim_{N\to\infty} \gamma_N \, 
\frac{\Cap_N (\ms E_x, \breve{\ms E}_x)}
{\mu_N(\ms E_x)}\;=\; \sum_{\eta\in \ms E_x} m_x(\eta) 
\sum_{\xi\in \breve{\ms E}_x} r_{\ms E}(\eta, \xi) \;\in\;\bb R_+
\;, 
\end{equation*}
which completes the proof of the first claim of the assertion.

By \eqref{02} and by definition of $r_{\ms E}(\eta, \xi)$,
\begin{equation*}
\sum_{x\in S} \sum_{\eta\in\ms E_x} \sum_{\xi \in \breve{\ms E}_x}
r_{\ms E}(\eta,\xi) \;=\;1\;,
\end{equation*}
so that
\begin{equation*}
\ell \;=\; \sum_{x\in S} \ell_x \;=\;
\sum_{x\in S} \sum_{\eta\in \ms E_x} m_x(\eta) 
\sum_{\xi\in \breve{\ms E}_x} r_{\ms E}(\eta, \xi) \;\ge\;
\min_{x\in S}\min_{\eta\in \ms E_x} m_x(\eta) \;>\; 0\;,
\end{equation*}
which is the second claim of the assertion.
\end{proof}

It follows from Assertion \ref{as04} that the time-scale $\gamma_N$ is
of the same order of $\theta_N$ in the sense that $\gamma_N/\theta_N$
converges as $N\uparrow\infty$:
\begin{equation}
\label{07}
\lim_{N\to\infty} \frac {\gamma_N}{\theta_N} \;=\; 
\ell \;\in\; (0,\infty)\;.
\end{equation}

\smallskip\noindent{\bf 3. The average jump rate, condition (H1).} 
Denote by $r_N(\ms E_x,\ms E_y)$ the mean rate at which the trace
process jumps from $\ms E_x$ to $\ms E_y$:
\begin{equation}
\label{04}
r_N(\ms E_x,\ms E_y) \; = \; \frac{1}{\mu_N(\ms E_x)}
\sum_{\eta\in\ms E_x} \mu_N(\eta) 
\sum_{\xi\in\ms E_y} R^{\ms  E}_N(\eta,\xi) \;.
\end{equation}

Next lemma follows from \eqref{09}, \eqref{06} and \eqref{07}. 

\begin{lemma}
\label{s06}
For every $x\not =y\in S$,
\begin{equation*}
r_{\ms E}(x,y) \;:=\;  \lim_{N\to\infty} \theta_N\, r_N(\ms E_x,\ms E_y) \; = \;
\frac 1\ell\, \sum_{\eta\in\ms E_x} m_x(\eta) 
\sum_{\xi\in\ms E_y} r_{\ms E} (\eta,\xi) \;\in\; \bb R_+\;
\end{equation*}  
\end{lemma}

\smallskip\noindent{\bf 4. Inside the metastable sets, condition
  (H2).} Next assertion shows that condition (H2) is in force.

\begin{asser}
\label{as05}
For every $x\in S$ for which $\ms E_x$ is not a singleton and for all
$\eta\not =\xi\in \ms E_x$, there exist constants $0<c_0<C_0<\infty$
such that
\begin{equation*}
c_0 \;\le\; \liminf_{N\to\infty}
\alpha_N\, \frac{\Cap_N(\eta,\xi)}{\mu_N(\ms E_x)}
\;\le\; \limsup_{N\to\infty}
\alpha_N\, \frac{\Cap_N(\eta,\xi)}{\mu_N(\ms E_x)}
\;\le\;  C_0 \;.
\end{equation*}
\end{asser}

\begin{proof}
Fix $x\in S$ for which $\ms E_x$ is not a singleton, and $\eta\not
=\xi\in \ms E_x$. On the one hand, by definition of the capacity
\begin{equation*}
\alpha_N\, \frac{\Cap_N(\eta,\xi)}{\mu_N(\ms E_x)} \;\le\;
\frac {\mu_N(\eta)}{\mu_N(\ms E_x)}\, 
\alpha_N\,\lambda_N(\eta)\;. 
\end{equation*}
By \eqref{01} and \eqref{09}, the right hand side converges to
$\lambda(\eta) m_x(\eta)<\infty$, which proves one of the
inequalities.

On the other hand, as $\ms E_x$ is an equivalent class which is not a
singleton, $\lambda(\zeta)>0$ for all $\zeta\in \ms E_x$, or, in other
words, $\ms E_x \subset E_0$.  Since $\eta \sim \xi$, there exists a
path $(\eta=\eta_0, \dots, \eta_n=\xi)$ such that
$R(\eta_i,\eta_{i+1})>0$ for $0\le i<n$. Since,
\begin{equation*}
\bb P_\eta \big[ H_{\xi} < H^+_{\eta} \big] \;\ge\;
p_N(\eta,\eta_1) \cdots p_N(\eta_{n-1},\xi)\;,
\end{equation*}
in view of the formula \eqref{28} for the capacity, we have that
\begin{equation*}
\alpha_N\, \frac{\Cap_N(\eta,\xi)}{\mu_N(\ms E_x)}
\;\ge\; \frac {\mu_N(\eta)}{\mu_N(\ms E_x)}\, 
\alpha_N\, \lambda_N(\eta)\, p_N(\eta,\eta_1) \cdots
p_N(\eta_{n-1},\xi)\;. 
\end{equation*}
The right hand side converges to $m_x(\eta) \lambda(\eta)
p(\eta,\eta_1) \cdots p(\eta_{n-1},\xi)>0$, which completes the proof
of the assertion.
\end{proof}

\smallskip\noindent{\bf 5. Condition (H3) holds.} 
To complete the proof of Theorem \ref{mt1} it remains to show that the
chain $\eta^N_t$ spends a negligible amount of time on the set
$\Delta$ in the time scale $\theta_N$.

\begin{lemma}
\label{s17}
For every $t>0$,
\begin{equation*}
\lim_{N\to \infty} \max_{\eta\in E} \, \bb E_\eta \Big[
\int_0^t \mb 1\{ \eta^N_{s\theta_N} \in \Delta\} \, ds  \Big] \;=\; 0\;. 
\end{equation*}
\end{lemma}

\begin{proof}
Since $\alpha_N/\theta_N \to 0$, a change of
variables in the time integral and the Markov property show that for
every $\eta\in E$, for every $T>0$ and for every $N$ large enough,
\begin{equation*}
\bb E_\eta \Big[ \int_0^t 
\mb 1\{ \eta^N_{s \theta_N} \in \Delta \} 
\, ds \Big] \;\le\; \frac {2t}T\,
\max_{\xi\in E}  \, \bb E_\xi \Big[ \int_0^T
\mb 1\{ \eta^N_{s \alpha_N} \in \Delta \} 
\, ds \Big]\;.
\end{equation*}
Note that the process on the right hand side is
speeded up by $\alpha_N$ instead of $\theta_N$.

We estimate the expression on the right hand side of the previous
formula. We may, of course, restrict the maximum to $\Delta$. Let
$T_1$ be the first time the chain $\eta^{N}_t$ hits $\ms E$ and let
$T_2$ be the time it takes for the process to return to $\Delta$
after $T_1$:
\begin{equation*}
T_1 \;=\; H_{\ms E}\; , \quad
T_2 \;=\; \inf \big\{ s> 0 : \eta^N_{T_1+s} \in \Delta\big\} \; .
\end{equation*}

Fix $\eta\in \Delta$ and note that
\begin{equation}
\label{22}
\begin{split}
&  \bb E_\eta \Big[ \frac 1T \int_0^T
\mb 1\{ \eta^N_{s \alpha_N} \in \Delta \} 
\, ds \Big]  \\
& \qquad \;\le\; 
\bb P_\eta \big[ T_1 > t_0 \alpha_N \big]\; +\; 
\bb P_\eta \big[ T_{2} < T \alpha_N \big] 
\;+\; \frac{t_0}T
\end{split}
\end{equation}
for all $t_0>0$ because the time average is bounded by $1$ and because
on the set $\{T_1 \le t_0 \alpha_N\}\cap \{ T_{2} \ge T \alpha_N\}$
the time average is bounded by $t_0/T$. By Assertion \ref{as18} below,
the first term on the right hand side vanishes as $N\uparrow\infty$
and then $t_0\uparrow\infty$. On the other hand, by the strong Markov
property, the second term is bounded by $\max_{\xi\in \ms E} \bb P_\xi
[ H_{\Delta} \le T \alpha_N ]$. By definition of the set $\ms E$, for
every $\eta\in\ms E$ and every $\xi\in\Delta$, $\alpha_N R_N(\eta,\xi)
\to 0$ as $N\uparrow\infty$. This shows that for every $T>0$ the
second term on the right hand side of \eqref{22} vanishes as
$N\uparrow\infty$, which completes the proof of the lemma.
\end{proof}

\begin{asser}
\label{as18}
For every $\eta\in \Delta$,
\begin{equation*}
\lim_{t\to\infty}
\limsup_{N\to\infty} \bb P_\eta \big[ H_{\ms E} \ge t \alpha_N
\big]\;=\; 0\;.
\end{equation*}
\end{asser}

\begin{proof}
Recall that we denote by $X_R(t)$ the continuous-time Markov
chain on $E$ which jumps from $\eta$ to $\xi$ at rate $R(\eta,\xi) =
\lim_{N} \alpha_N R_N(\eta,\xi)$. Note that the set $\ms E$ consists
of recurrent points for the chain $X_R(t)$, while points in $\Delta$
are transient. Since the jump rates converge, the chain
$\eta^N_{t\alpha_N}$ converges in the Skorohod topology to $X_R(t)$.
Therefore, for all $t>0$, $\eta\in \Delta$,
\begin{equation*}
\limsup_{N\to\infty} \bb P_\eta \big[ H_{\ms E} \ge
t\,\alpha_N  \big] \;\le \; \mb P_\eta \big[ H_{\ms E} \ge t \big]\;,
\end{equation*}
where $\mb P_\eta$ stands for the law of the chain $X_R(t)$ starting
from $\eta$. Since the set of recurrent points for $X_R(t)$ is equal to
$\ms E = \Delta^c$, the previous probability vanishes as
$t\uparrow\infty$.  
\end{proof}

We conclude this section with an observation concerning the capacities
of the metastable sets $\ms E_x$.

\begin{asser}
\label{as08}
The sequences $\{\Cap_N (\ms E_x, \breve{\ms E}_x)/\mu_N(\ms E_x) :
N\ge 1\}$, $x\in S$, are ordered.
\end{asser}

\begin{proof}
Fix $x\in S$. By \eqref{12} applied to $A= \ms E_x$, $B=\breve{\ms E}_x$,
\begin{equation*}
\Cap_N (\ms E_x, \breve{\ms E}_x)
\;=\; \sum_{\eta\in \ms E_x} \mu_N(\eta) \sum_{\xi\in \breve{\ms E}_x}
R^{\ms E}_N (\eta, \xi) \;.  
\end{equation*}
The claim of the assertion follows from this identity, from
Assertion \ref{as01} and from \eqref{09}.
\end{proof}

\section{Proof of Theorem \ref{mt2}}

Theorem \ref{mt2} is proved in several steps.

\smallskip\noindent {\bf 1. The measure of configurations in $\ms
  G_a$.} We assumed in (H0) that all configurations in a set $\ms F_x$
have measure of the same order. We prove below in Assertion \ref{as07}
that a similar property holds for the sets $\ms G_a$.

Let 
\begin{equation*}
\lambda^{\ms F}_N(\ms F_x) \;=\; \sum_{y: y\not = x} r^{\ms F}_N(\ms
F_x,\ms F_y)\;, \quad
p^{\ms F}_N(\ms F_x,\ms F_y)  \;=\; \frac{r^{\ms F}_N(\ms F_x,\ms
  F_y)}{\lambda^{\ms F}_N(\ms F_x)}\;\; \text{if }\;
\lambda^{\ms F}_N(\ms F_x)>0 \;.
\end{equation*}
Denote by $P_0$ the subset of points in $P$ such that $\lambda_{\ms
  F}(x) = \sum_{y\not = x} r_{\ms F}(x,y)>0$. For all $x\in P_0$ let
$p_{\ms F}(x,y)=r_{\ms F}(x,y)/\lambda_{\ms F}(x)$.  It follows from
assumption (H1) that for all $x$, $z$ in $P$, $y\in P_0$,
\begin{equation}
\label{b01}
\lim_{N\to\infty} \beta_N \, \lambda^{\ms F}_N(\ms F_x)\;=\lambda_{\ms F}(x)\;,
\quad \lim_{N\to\infty} p^{\ms F}_N(\ms F_y,\ms F_z) \;=\; p_{\ms F}(y,z)\;.
\end{equation}

Recall that $X_{\ms F}(t)$ is the $P$-valued Markov chain which jumps
from $x$ to $y$ at rate $r_{\ms F}(x,y)$.  Denote by $C_a$, $a\in P_1
=\{1, \dots, \mf q_1\}$, the equivalent classes of the Markov chain
$X_{\ms F}(t)$, and let $\ms C_a = \cup_{x\in C_a} \ms F_x$. All
configurations in a set $\ms C_a$ have probability of the same order.

\begin{asser}
\label{as07}
For all equivalent classes $C_a$, $a\in P_1$, and for all $\eta\not =
\xi\in \ms C_a$, there exists $m(\eta,\xi)\in (0,\infty)$ such that
\begin{equation*}
\lim_{N\to\infty} \frac{\mu_N(\eta)}{\mu_N(\xi)} \; = \; m(\eta,\xi)\;.
\end{equation*}
\end{asser}

\begin{proof}
The argument is very close to the one of Assertion \ref{as03}
Denote by $\bar X_N (t)$ the chain $\eta^{\ms F} (t)$ in which each
set $\ms F_x$ has been collapsed to a point. We refer to the Section 3
of \cite{gl14} for a precise definition of the collapsed chain and for
the proof of the results used below.

The chain $\bar X_N(t)$ takes value in the set $P$, its jump rate
from $x$ to $y$, denoted by $\bar r_N(x,y)$, is equal to $r^{\ms
  F}_N(\ms F_x, \ms F_y)$ introduced in \eqref{b04}, and its unique
invariant probability measure, denoted by $\bar{\mu}_N(x)$, is given by
$\bar \mu_N(x) = \mu_N(\ms F_x)/\mu_N(\ms F)$.

Fix an equivalent class $C_a$ and $\eta\not = \xi\in \ms C_a$. If
$\eta$ and $\xi$ belong to the same set $\ms F_x$, the claim follows
from Assumption (H0). Suppose that $\eta\in\ms F_x$, $\xi\in\ms F_y$
for some $x\not =y\in C_a$.  By assumption, there exists a path
$(x=x_0, \dots, x_n=y)$ such that $r_{\ms F}(x_i,x_{i+1})>0$ for $0\le
i<n$.

Denote by $\bar\lambda_N(x)$, $x\in P$, the holding rates of the
collapsed chain $\bar X_N (t)$, and by $\bar p_N(x,y)$, $x\not = y\in
P$, the jump probabilities.  Since $\bar\mu_N$ is the invariant
probability measure for the collapsed chain,
\begin{align*}
\bar \lambda_N(y) \, \bar \mu_N(y) \; &=\; \sum_{z_0, z_1, \dots, z_{n-1}\in P} 
\bar \mu_N(z_0) \, \bar \lambda_N(z_0) \, \bar p_N(z_0,z_1) \cdots 
\bar p_N(z_{n-1}, y) \\
&\ge \; \bar \mu_N(x_0) \, \bar\lambda_N(x_0) \, \bar p_N(x_0,x_1) \cdots 
\bar p_N(x_{n-1}, y)\;.
\end{align*}
Therefore,
\begin{equation*}
\frac{\bar\mu_N(y)}{\bar\mu_N(x)} \;\ge\; 
\frac{\bar\lambda_N(x)}{\bar\lambda_N(y)} \, 
\bar p_N(x,x_1) \cdots  \bar p_N(x_{n-1}, y) \;.  
\end{equation*}
Since $r_{\ms F}(x_i,x_{i+1})>0$ for $0\le i<n$, by \eqref{b01}, $\bar
p_N(x_{i}, x_{i+1})$ converges to $p_{\ms F}(x_{i}, x_{i+1})>0$.  For
the same reason, $\beta_N \bar\lambda_N(x)= \beta_N \lambda^{\ms
  F}_N(\ms F_x)$ converges to $\lambda_{\ms F}(x)\in (0,\infty)$. As
$y$ and $x$ share the same properties, inverting their role we obtain
that $\beta_N \bar\lambda_N(y)$ converges to $\lambda_{\ms F}(y)\in
(0,\infty)$. In conclusion,
\begin{equation*}
\liminf_{N\to\infty} \frac{\bar\mu_N(x)}{\bar\mu_N(y)} \;>\; 0\;.
\end{equation*}
Replacing $x$ by $y$ we obtain that $\liminf \bar
\mu_N(y)/\bar\mu_N(x)>0$. By \cite{gl14}, $\bar\mu_N(z) = \mu_N(\ms
F_z)$, $z\in P$. To complete the proof it remains to recall the
statement of Lemma \ref{as02} and Assumption (H0).
\end{proof}

By the previous assertion for every $a\in Q$ and $\eta\in\ms G_a$,
\begin{equation}
\label{b11}
m^*_a(\eta) \;:=\; \lim_{N\to\infty}  \frac{\mu_N(\eta)}{\mu_N(\ms
  G_a)} \;\in\; (0,1]\;. 
\end{equation}
Thus, assumption (H0) holds for the partition $\{\ms G_1, \dots, \ms
G_{\mf q}, \Delta_{\ms G}\}$.

\smallskip\noindent {\bf 2. The time scale.} We prove in this
subsection that the time-scale $\beta^+_N$ introduced in \eqref{b05}
is much longer than $\beta_N$.

\begin{asser}
\label{as09}
We have that
\begin{equation*}
\lim_{N\to\infty} \frac{\beta_N}{\beta^+_N}\;=\; 0\;.
\end{equation*}
\end{asser}

\begin{proof}
We have to show that
\begin{equation*}
\lim_{N\to\infty} \beta_N \frac{\Cap_N (\ms G_a, \breve{\ms G}_a)} 
{\mu_N(\ms G_a)} \;=\;0
\end{equation*}
for each $a\in Q$. Fix $a\in Q$ and recall from \eqref{b12} the
definition of the set $\ms G_a$. Since $G_a$ is recurrent class for
the chain $X_{\ms F}(t)$, $r_{\ms F}(x,y)=0$ for all $x\in G_a$, $y\in
P\setminus G_a$. By definition of the capacity,
\begin{align*}
\frac{\Cap_N (\ms G_a, \breve{\ms G}_a)}{\mu_N(\ms G_a)} 
\; & =\; \sum_{\eta\in \ms G_a} \frac{\mu_N(\eta)}{\mu_N(\ms G_a)}
\lambda_N (\eta) \, \bb
P_{\eta} \big[ H_{\breve{\ms G}_a} < H^+_{\ms G_a} \big] \\
\; & \le \; \sum_{\eta\in \ms G_a} \frac{\mu_N(\eta)}{\mu_N(\ms G_a)}
\lambda_N (\eta) \, \bb
P_{\eta} \big[ H_{\ms F \setminus \ms G_a} < H^+_{\ms G_a} \big] \;.  
\end{align*}
By \cite[Proposition 6.1]{bl2}, this sum is equal to
\begin{equation*}
\sum_{\eta\in \ms G_a} \frac{\mu_N(\eta)}{\mu_N(\ms G_a)}
\, \sum_{\xi \in \ms F \setminus \ms G_a} R^{\ms F}_N (\eta,\xi)
\;=\; \sum_{x\in G_a} \frac{\mu_N(\ms F_x)}{\mu_N(\ms G_a)}
\sum_{y\in P \setminus G_a} r^{\ms F}_N(x,y)\;.
\end{equation*}
Since $r_{\ms F}(x,y)=0$ for all $x\in G_a$, $y\in P\setminus
G_a$, by assumption (H1) the previous sum multiplied by $\beta_N$
converges to $0$ as $N\uparrow\infty$.
\end{proof}

\smallskip\noindent
{\bf 3. Condition (H1) is fulfilled by the partition $\{\ms G_1,
  \dots, \ms G_{\mf q}, \Delta_{\ms G}\}$.}
We first obtain an alternative formula for the time-scale
$\beta^+_N$. The arguments and the ideas are very similar to the ones
presented in the previous section. Let 
\begin{equation*}
\frac 1{\gamma_N} \;=\; \sum_{a\in Q} \sum_{\eta\in \ms G_a} 
\sum_{\xi\in \breve{\ms G}_a} R^{\ms G}_N (\eta, \xi) \;.
\end{equation*}
By Assertion \ref{as01}, for all $a\in Q$, $\eta\in \ms G_a$, $\xi\in
\breve{\ms G}_a$, with the convention adopted in condition (H1) of
Section \ref{not},
\begin{equation}
\label{b06}
r_{\ms G} (\eta,\xi) \;:=\; \lim_{N\to\infty} 
\gamma_N \, R^{\ms G}_N (\eta, \xi) \;\in\; [0,1]\;.
\end{equation}

\begin{asser}
\label{as12}
For all $a\in Q$, 
\begin{equation*}
\hat \lambda_{\ms G} (a) \;:=\; \lim_{N\to\infty} \gamma_N \, 
\frac{\Cap_N (\ms G_a, \breve{\ms G}_a)}
{\mu_N(\ms G_a)}\;\in\; \bb R_+ \;. \quad\text{Moreover}\;,\;\;
\hat \lambda_{\ms G} \;=\; \sum_{a\in Q} \hat \lambda_{\ms G} (a) \;>\;0 \;.
\end{equation*}
\end{asser}

\begin{proof}
Fix $a\in Q$. By \eqref{12}, applied to $A= \ms G_a$, $B=\breve{\ms
  G}_a$, by \eqref{b11} and by \eqref{b06}, 
\begin{equation*}
\lim_{N\to\infty} \gamma_N\, 
\frac{\Cap_N (\ms G_a, \breve{\ms G}_a)}
{\mu_N(\ms G_a)}\;=\; \sum_{\eta\in \ms G_a} m^*_a(\eta) 
\sum_{\xi\in \breve{\ms G}_a} r_{\ms G}(\eta, \xi) \;\in\;\bb R_+
\;, 
\end{equation*}
which completes the proof of the first claim of the assertion.

By definition of $\gamma_N$ and by definition of $r_{\ms G}(\eta,
\xi)$,
\begin{equation*}
\sum_{a\in Q} \sum_{\eta\in\ms G_a} \sum_{\xi \in \breve{\ms G}_a}
r_{\ms G}(\eta,\xi) \;=\;1\;,
\end{equation*}
so that
\begin{equation*}
\hat \lambda_{\ms G} \;=\; \sum_{a\in Q} \hat \lambda_{\ms G} (a) \;=\;
\sum_{a\in Q} \sum_{\eta\in \ms G_a} m^*_a(\eta) 
\sum_{\xi\in \breve{\ms G}_a} r_{\ms G}(\eta, \xi) \;\ge\;
\min_{a\in Q}\min_{\eta\in \ms G_a} m^*_a(\eta) \;>\; 0\;,
\end{equation*}
which is the second claim of the assertion.
\end{proof}

It follows from the previous assertion that the time-scale $\gamma_N$
is of the same order of $\beta^+_N$:
\begin{equation}
\label{b07}
\lim_{N\to\infty} \frac {\gamma_N}{\beta^+_N} \;=\; 
\hat \lambda_{\ms G} \;\in (0,\infty) \;.
\end{equation}

Denote by $r^{\ms G}_N(\ms G_a,\ms G_b)$ the mean rate at which the
trace process jumps from $\ms G_a$ to $\ms G_b$:
\begin{equation}
\label{b03}
r^{\ms G}_N(\ms G_a,\ms G_b) \; := \; \frac{1}{\mu_N(\ms G_a)}
\sum_{\eta\in\ms G_a} \mu_N(\eta) 
\sum_{\xi\in\ms G_b} R^{\ms  G}_N(\eta,\xi) \;.
\end{equation}

\begin{lemma}
\label{s07}
For every $a\not =b\in Q$,
\begin{equation*}
r_{\ms G}(a,b) \;:=\;  \lim_{N\to\infty} \beta^+_N \, 
r^{\ms G}_N(\ms G_a,\ms G_b) \; = \;
\frac 1{\hat \lambda_{\ms G}} \, \sum_{\eta\in\ms G_a} m^*_a(\eta) 
\sum_{\xi\in\ms G_b} r_{\ms G} (\eta,\xi) \;\in\; \bb R_+\;
\end{equation*}
Moreover,
\begin{equation*}
\sum_{a \in Q} \sum_{b: b\not =a} r_{\ms G}(a,b) \;=\; 1\;.
\end{equation*}
\end{lemma}

\begin{proof}
The first claim of this lemma follows from \eqref{b11}, \eqref{b06}
and \eqref{b07}. On the other hand, by the explicit formula for
$r_{\ms G}(a,b)$ and by the formula for $\hat \lambda_{\ms G} (a)$
obtained in the previous assertion,
\begin{equation*}
\sum_{a \in Q} \sum_{b: b\not =a} r_{\ms G}(a,b) \;=\;
\frac 1{\hat \lambda_{\ms G}} \,  \sum_{a \in Q} 
\sum_{\eta\in\ms G_a} m^*_a(\eta) \sum_{b: b\not =a}
\sum_{\xi\in\ms G_b} r_{\ms G} (\eta,\xi) 
\;=\;\frac 1{\hat \lambda_{\ms G}} \,  \sum_{a \in Q} 
\hat \lambda_{\ms G} (a) \;.
\end{equation*}
This expression is equal to $1$ by definition of $\hat \lambda_{\ms
  G}$.  
\end{proof}

\smallskip\noindent
{\bf 4. Condition (H2) is fulfilled by the partition $\{\ms G_1,
  \dots, \ms G_{\mf q}, \Delta_{\ms G}\}$.}
The proof of condition (H2) is based on the next assertion.

\begin{asser}
\label{as13}
For every $a\in Q$ for which $\ms G_a$ is not a singleton and for all
$\eta\not =\xi\in \ms G_a$, 
\begin{equation*}
\liminf_{N\to\infty}
\beta_N \frac{\Cap_N(\eta,\xi)}{\mu_N(\ms G_a)}
\;>\; 0 \;.
\end{equation*}
\end{asser}

\begin{proof}
Throughout this proof $c_0$ represents a positive real number
independent of $N$ and which may change from line to line.  Fix $a\in
Q$ for which $\ms G_a$ is not a singleton, and $\eta\not =\xi\in \ms
G_a$. By definition, $\ms G_a = \cup_{x\in G_a} \ms F_x$. If $\eta$
and $\xi$ belongs to the same $\ms F_x$, the result follows from
assumption (H2) and from Assertion \ref{as07}.

Fix $\eta\in \ms F_x$ and $\xi\in \ms F_y$ for some $x\not =y$, $\ms
F_x \cup \ms F_y\subset \ms G_a$.  Recall that we denote by
$\Cap^s_N(\ms A, \ms B)$ the capacity between two disjoint subsets
$\ms A$, $\ms B$ of $E$ with respect to the reversible chain
introduced in \eqref{30}.

Since $G_a$ is a recurrent class for the chain $X_{\ms F}(t)$, there
exists a sequence $(x=x_0, x_1, \dots, x_n=y)$ such that $r_{\ms
  F}(x_i,x_{i+1})>0$ for $0\le i<n$. in view of assumptions (H0) and
(H1), there exist $\xi_i\in\ms F_{x_i}$, $\eta_{i+1}\in\ms
F_{x_{i+1}}$ such that $\beta_N R^{\ms F}_N (\xi_i, \eta_{i+1})\ge
c_0$. Therefore, by Corollary \ref{s10} and \eqref{19},
\begin{equation}
\label{20}
\beta_N \, \Cap^s_N (\xi_i, \eta_{i+1}) \;\ge\; \frac{\beta_N}{2|E|} \, 
\Cap_N (\xi_i, \eta_{i+1}) \;\ge\; c_0 \, \mu_N(\xi_i) \;,
\end{equation}
so that, by \eqref{24}, $\beta_N \, \bs c^s_N (\xi_i, \eta_{i+1}) \ge
c_0 \, \mu_N(\xi_i)$.

Since the configuration $\eta$ and $\xi_0$ belongs to the same set
$\ms F_x$, by assumption (H2), $\beta^-_N \Cap_N(\eta,\xi_0)/\mu_N(\ms
F_x) \ge c_0$. A similar assertion holds for the pair of
configurations $\eta_i$, $\xi_i$, $1\le i<n$, and for the pair
$\eta_n$, $\xi$. Hence, if we set $\eta_0=\eta$, $\xi_n=\xi$, by
Corollary \ref{s10} and \eqref{24}, we have that
\begin{equation*}
\beta^-_N \bs c^s_N(\eta_i,\xi_i) \;\ge\; c_0 \mu_N(\ms F_{x_i})\;.
\end{equation*}

By \eqref{b11}, we may replace $\mu_N(\ms F_{x_i})$ by $\mu_N(\ms
G_{a})$ in the previous inequality, and $\mu_N(\xi_i)$ by $\mu_N(\ms
G_{a})$ in \eqref{20}. By \eqref{26},
\begin{equation*}
\bs c^s_N(\eta,\xi) \;\ge\; \min_{0\le i <n} \min\big\{ \bs c^s_N(\eta_i,
\xi_i), \bs c^s_N(\xi_i, \eta_{i+1}), \bs c^s_N(\eta_n, \xi_n)\big\}\;.
\end{equation*}
Since $\beta^-_N\ll \beta_N$, it follows from the previous estimates
that $\beta_N \bs c^s_N(\eta,\xi) \ge c_0 \mu_N(\ms G_{a})$. To complete
the proof, it remains to recall that, by Corollary \ref{s10} and
\eqref{24}, $\Cap_N(\eta,\xi) \ge \Cap^s_N(\eta,\xi) \ge c_0 \bs
c^s_N(\eta,\xi)$.
\end{proof}

\smallskip\noindent {\bf 5. Condition (H3) is fulfilled by the partition $\{\ms G_1,
  \dots, \ms G_{\mf q}, \Delta_{\ms G}\}$.} Lemma \ref{s15} shows that it is enough to prove
condition (H3) for the trace process $\eta^{\ms F}(t)$.

\begin{lemma}
\label{s15}
Assume that 
\begin{equation*}
\lim_{N\to \infty} \max_{\eta\in \mc F} \bb E_\eta \Big[ \int_0^t 
\mb 1\{ \eta^{\ms F}_{s \beta^+_N} \in \Delta_* \} \, ds
\Big]\;=\; 0\;,
\end{equation*}
where $\Delta_* = \cup_{x\in G_{\mf q+1}} \ms F_x$ has been introduced in
\eqref{b12}.  Then,
\begin{equation*}
\lim_{N\to \infty} \max_{\eta\in E} \bb E_\eta \Big[ \int_0^t 
\mb 1\{ \eta^{N}_{s \beta^+_N} \in \Delta_{\ms G} \} \, ds \Big]\;=\; 0\;.
\end{equation*}
\end{lemma}

\begin{proof}
Fix $\eta \in E$. Since $\Delta_{\ms G} = \Delta_* \cup \Delta_{\ms F}$,
\begin{align*}
& \bb E_\eta \Big[ \int_0^t \mb 1\{ \eta_{s \beta^+_N} 
\in \Delta_{\ms F} \cup \Delta_* \} \, ds \Big] \\ 
&\qquad \;\le\; \bb E_\eta \Big[ \int_0^t \mb 1\{ \eta_{s \beta^+_N} 
\in \Delta_{\ms F}\} \, ds \Big] \;+\; \max_{\xi\in \ms F}\,
\bb E_\xi \Big[ \int_0^t \mb 1\{ \eta^{\ms F}_{s \beta^+_N} 
\in \Delta_* \} \, ds \Big]\;.    
\end{align*}
The second term vanishes as $N\uparrow\infty$ by assumption.
The first one is bounded by
\begin{equation*}
\frac {\beta_N}{\beta^+_N} \sum_{n=0}^{[\beta^+_N/\beta_N]} 
\bb E_\eta \Big[ \int_{nt}^{(n+1)t} \mb 1\{ \eta_{s \beta_N} 
\in \Delta_{\ms F}\} \, ds \Big]\;,
\end{equation*}
where $[r]$ stands for the integer part of $r$. By the Markov
property, this expression is bounded above by
\begin{equation*}
2\, \max_{\xi\in E} 
\bb E_\xi \Big[ \int_{0}^{t} \mb 1\{ \eta_{s \beta_N} 
\in \Delta_{\ms F}\} \, ds \Big]\; ,
\end{equation*}
which vanishes as $N\uparrow\infty$ by assumption (H3).
\end{proof}

To prove that condition (H3) is fulfilled by the partition $\{\ms G_1,
\dots, \ms G_{\mf q}, \Delta_{\ms G}\}$ it remains to show that the
assumption of the previous lemma is in force.  The proof of this claim
relies on the next assertion. Denote by $\bb P^{\ms F}_\eta$ the
probability measure on $D(\bb R_+, \ms F)$ induced by the trace chain
$\eta^{\ms F}_t$ starting from $\eta$.

\begin{asser}
\label{as17}
For every $\eta\in \Delta_*$,
\begin{equation*}
\lim_{t\to\infty}
\limsup_{N\to\infty} \bb P^{\ms F}_\eta \big[ H_{\ms G} \ge t \beta_N
\big]\;=\; 0\;.
\end{equation*}
\end{asser}

\begin{proof}
Fix $\eta\in \ms F_x \subset \Delta_*$.  Since the partition $\ms F_1,
\dots, \ms F_{\mf p}$, $\Delta_{\ms F}$ satisfy the conditions (H1)--(H3), by
Proposition \ref{s04}, starting from $\eta$ the process $X_N(t) =
\psi_{\ms F}(\eta^{\ms F}_{t\beta_N})$ converges in the Skorohod
topology to the Markov chain $X_{\ms F}(t)$ on $P=\{1, \dots, \mf p\}$
which starts from $x$ and which jumps from $y$ to $z$ at rate $
r_{\ms F}(y,z)$.  Therefore,
\begin{equation*}
\limsup_{N\to\infty} \bb P^{\ms F}_\eta \big[ H_{\ms G} \ge
t\,\beta_N  \big] \;\le \; \mb P_x \big[ H_R \ge t \big]\;,
\end{equation*}
where $\mb P_x$ represents the distribution of the chain $X_{\ms
  F}(t)$ starting from $x$ and $R = \cup_{1\le a\le \mf q} G_a$.
Since $R$ corresponds to the set of recurrent points of the chain
$X_{\ms F}(t)$, the previous expression vanishes as $t\uparrow\infty$.
\end{proof}

\begin{lemma}
\label{s16}
For all $t>0$,
\begin{equation*}
\lim_{N\to \infty} \max_{\eta\in \ms F} \bb E_\eta \Big[ \int_0^t 
\mb 1\{ \eta^{\ms F}_{s \beta^+_N} \in \Delta_* \} 
\, ds \Big]\;=\; 0\;.
\end{equation*}
\end{lemma}

\begin{proof}
Since $\beta_N/ \beta^+_N\to 0$, a change of variables in the time
integral, similar to the one performed in the proof of Lemma
\ref{s15}, and the Markov property show that for every $\eta\in \ms
F$, every $T>0$ and every $N$ large enough,
\begin{equation*}
\bb E_\eta \Big[ \int_0^t 
\mb 1\{ \eta^{\ms F}_{s \beta^+_N} \in \Delta_* \} 
\, ds \Big] \;\le\; \frac {2t}T\,
\max_{\xi\in \ms F}  \, \bb E_\xi \Big[ \int_0^T
\mb 1\{ \eta^{\ms F}_{s \beta_N} \in \Delta_* \} 
\, ds \Big]\;.
\end{equation*}
Note that the process on the right hand side is speeded up by
$\beta_N$ instead of $\beta^+_N$.

We estimate the expression on the right hand side of the previous
formula. We may, of course, restrict the maximum to $\Delta_*$. Let
$T_1$ be the first time the trace process $\eta^{\ms F}_t$ hits $\ms
G$ and let $T_2$ be the time it takes for the process to return to
$\Delta_*$ after $T_1$:
\begin{equation*}
T_1 \;=\; H_{\ms G}\; , \quad
T_2 \;=\; \inf \big\{ s> 0 : \eta^{\ms F}_{T_1+s} \in \Delta_*\big\} \; .
\end{equation*}

Fix $\eta\in \Delta_*$ and note that
\begin{equation*}
\begin{split}
&  \bb E_\eta \Big[ \frac 1T \int_0^T
\mb 1\{ \eta^{\ms F}_{s \beta_N} \in \Delta_* \} 
\, ds \Big]  \\
& \qquad \;\le\; 
\bb P^{\ms F}_\eta \big[ T_1 > t_0 \beta_N \big]\; +\; 
\bb P^{\ms F}_\eta \big[ T_{2} \le T \beta_N \big] 
\;+\; \frac{t_0}T
\end{split}
\end{equation*}
for all $t_0>0$. By Assertion \ref{as17}, the first term on the right
hand side vanishes as $N\uparrow\infty$ and then
$t_0\uparrow\infty$. On the other hand, by the strong Markov property,
the second term is bounded by $\max_{\xi\in \ms G} \bb P^{\ms F}_\xi [
H_{\Delta_*} \le T \beta_N ]$. Since, by Proposition \ref{s04}, the
process $\psi_{\ms F}(\eta^{\ms F}_{t\beta_N})$ converges in the
Skorohod topology to the Markov chain $X_{\ms F}(t)$,
\begin{equation*}
\limsup_{N\to\infty} \max_{\xi\in \ms G} \bb P^{\ms F}_\xi [ H_{\Delta_*} \le T
\beta_N ] \;\le\;
\max_{1\le a\le \mf q} \max_{x\in G_a} \mb P_x [ H_{G_{\mf q+1}} \le T ]\;,
\end{equation*}
where, as in the proof of the previous assertion, $\mb P_x$ represents
the distribution of the chain $X_{\ms F}(t)$ starting from $x$. Since
the sets $G_a$ are recurrent classes for the chain $X_{\ms F}(t)$,
$r_{\ms F}(x,y)=0$ for all $x\in \cup_{1\le a\le \mf q} G_a$, $y\in
G_{\mf q+1}$. Therefore, the previous probability is equal to $0$ for
all $T>0$, which completes the proof of the lemma.
\end{proof}


\begin{thebibliography}{99}

\bibitem{ag14} L. Avena, A. Gaudilli\`ere: On some random forests with
  determinantal roots. Arxiv arXiv:1310.1723 (2013)

\bibitem{bl2} J. Beltr\'an, C. Landim: Tunneling and metastability of
  continuous time Markov chains.  \emph{J. Stat. Phys.} {\bf 140},
  1065--1114 (2010).%

\bibitem{bl4} J. Beltr\'an, C. Landim: Metastability of reversible
  finite state Markov processes.  \emph{Stoch. Proc. Appl}. {\bf 121},
  1633--1677 (2011).%

\bibitem{bl7} J. Beltr\'an, C. Landim: Tunneling and metastability of
  continuous time Markov chains II. J. Stat. Phys.  {\bf 149},
  598--618 (2012).%

\bibitem{bl9} J. Beltr\'an, C. Landim: A Martingale approach to
  metastability, Probab. Th. Rel. Fields. {\bf 161}, 267--307 (2015).%

\bibitem{blm13} O. Benois, C. Landim, C. Mourragui: Hitting Times of
  Rare Events in Markov Chains. J. Stat. Phys. {\bf 153}, 967--990
  (2013).%

\bibitem{bg15} A. Bianchi, A. Gaudilli\`ere: Metastable states,
  quasi-stationary distributions and soft measures. To appear in
  Stoch. Proc. Appl. (2016)%



\bibitem{bh1} A. Bovier, F. den Hollander: {\sl Metastability: a
    potential-theoretic approach}. Grundlehren der mathematischen
  Wissenschaften {\bf 351}, Springer, Berlin, 2015.%

\bibitem{ce14} M. Cameron, E. Vanden-Eijnden: Flows in Complex
  Networks: Theory, Algorithms, and Application to Lennard--Jones
  Cluster Rearrangement. J. Stat. Phys. {\bf 156}, 427--454 (2014)


\bibitem{cg15} P, Chleboun, S. Grosskinsky: A dynamical transition and
  metastability in a size-dependent zero-range process J. Phys. A:
  Math. Theor. {\bf 48}, 055001 (2015)


\bibitem{cns14} E. Cirillo, F. Nardi, J. Sohier: Metastability for
  general dynamics with rare transitions: escape time and critical
  configurations. arXiv:1412.7923 (2014)%

\bibitem{ev06} W. E, E. Vanden-Eijnden: Towards a theory of transition
  paths. J. Stat. Phys. {\bf 123}, 503--523 (2006)%

\bibitem{fmnss15} R. Fernandez, F. Manzo, F. Nardi, E. Scoppola,
  J. Sohier: Conditioned, quasi-stationary, restricted measures and
  metastability. Ann. Appl. Probab. (2015)%

\bibitem{fmns15} R. Fernandez, F. Manzo, F. Nardi, E. Scoppola:
  Asymptotically exponential hitting times and metastability: a
  pathwise approach without reversibility. Electron. J. Probab. (2015)%

\bibitem{fw} M. I. Freidlin, A. D. Wentzell: Random perturbations of
  dynamical systems. Translated from the 1979 Russian original by
  Joseph Sz\"ucs. Second edition. Grundlehren der Mathematischen
  Wissenschaften [Fundamental Principles of Mathematical Sciences],
  260. Springer-Verlag, New York, 1998. %

\bibitem{gv12} D. Gabrielli, C. Valente: Which random walks are
  cyclic? ALEA, Lat. Am. J. Probab. Math.  Stat. {\bf 9}, 231--267
  (2012)

\bibitem{gl14} A. Gaudilli\`ere, C. Landim; A Dirichlet principle for
  non reversible Markov chains and some recurrence theorems.
  Probab. Theory Related Fields {\bf 158}, 55--89 (2014) %

\bibitem{klo12} T. Komorowski, C. Landim and S. Olla; {\it
    Fluctuations in Markov Processes, Time Symmetry and Martingale
    Approximation}, Grundlheren der mathematischen Wissenschaften {\bf
    345}, Springer-Verlag, Berlin, New York, (2012).

\bibitem{l-soft} C. Landim; A topology for limits of Markov chains.
  Stoch. Proc. Appl. {\bf 125}, 1058--1098 (2014) %

\bibitem{l14} C. Landim; Metastability for a non-reversible dynamics:
  the evolution of the condensate in totally asymmetric zero range
  processes.  Commun. Math. Phys. {\bf 330}, 1--32 (2014)

\bibitem{lv14} J. Lu,  E. Vanden-Eijnden: Exact dynamical
  coarse-graining without time-scale separation. J. Chem. Phys. {\bf
    141}, 044109 (2014)%

\bibitem{mnos04} F. Manzo, F. Nardi, E. Olivieri, E. Scoppola: On the
  essential features of metastability: tunnelling time and critical
  configurations. J. Stat. Phys. {\bf 115}, 591--642 (2004)

\bibitem{mo13} C. Maes, W. O’Kelly de Galway: A Low Temperature
  Analysis of the Boundary Driven Kawasaki
  Process. J. Stat. Phys. {\bf 153}, 991--1007 (2013)%

\bibitem{msv09} P. Metzner, Ch. Schuette, E. Vanden-Eijnden:
  Transition path theory for Markov jump processes. SIAM Multiscale
  Model. Simul. {\bf 7}, 1192--1219 (2009)%

\bibitem {m14} R. Misturini: Evolution of the ABC model among the
  segregated configurations in the zero-temperature limit.
  arXiv:1403.4981 (2014) %

\bibitem{os95} E. Olivieri, E. Scoppola: Markov Chains with
  Exponentially Small Transition Probabilities: First Exit Problem
  from a General Domain. I. The Reversible Case J. Stat. Phys. {\bf 79},
  613--647 (1995). %

\bibitem{os96} E. Olivieri, E. Scoppola: Markov Chains with
  Exponentially Small Transition Probabilities: First Exit Problem
  from a General Domain. II. The General Case J. Stat. Phys. {\bf 84},
  987--1041 (1996). %

\bibitem{ov1} E. Olivieri and M. E. Vares. {\em Large deviations and
    metastability}. Encyclopedia of Mathematics and its Applications,
  vol. 100. Cambridge University Press, Cambridge, 2005.%

\bibitem{s93} E. Scoppola. Renormalization group for Markov chains and
  application to metastability.  J. Stat. Phys. {\bf 73}, 83--121
  (1993). %

\bibitem{sekc15} A. Singer, R. Erban, I. G. Kevrekidis,
  R. R. Coifman: Detecting intrinsic slow variables in stochastic
  dynamical systems by anisotropic diffusion maps
  Proc. Natl. Acad. Sci. USA {\bf 106}, 16090-16095 (2009)




\end{thebibliography}
\end{document}